\documentclass[10pt,twoside,a4paper]{article}

\usepackage[utf8]{inputenc}
\RequirePackage[l2tabu, orthodox]{nag}

\usepackage[american]{babel}
\usepackage[
			hmargin=1.85cm,
			vmargin=1.85cm,
			nomarginpar]
			{geometry}
\usepackage[fleqn]{amsmath}
\usepackage{amssymb,amsthm}
\usepackage{xspace} 
\usepackage[dvipsnames]{xcolor}
\usepackage{todonotes}
\usepackage{dsfont}
\usepackage{microtype}
\usepackage{appendix}
\DisableLigatures[f]{encoding = *}

\newcommand{\myName}{Alessandro~Zocca\xspace}
\newcommand{\myUni}{CWI\xspace}

\definecolor{dred}{RGB}{220,0,0}
\definecolor{halfgray}{gray}{0.55}
\definecolor{webgreen}{rgb}{0,.5,0}
\definecolor{webbrown}{rgb}{.6,0,0}
\usepackage[hyperfootnotes=false,pdfpagelabels]{hyperref}
\hypersetup{%
    colorlinks=true, linktocpage=true, pdfborder={0 0 0}, pdfstartpage=3, pdfstartview=FitV,%
    breaklinks=true, pdfpagemode=UseNone, pageanchor=true, pdfpagemode=UseOutlines,%
    plainpages=false, bookmarksnumbered, bookmarksopen=true, bookmarksopenlevel=1,%
    hypertexnames=true, pdfhighlight=/O,
    urlcolor=webbrown, linkcolor=RoyalBlue, citecolor=webgreen, 
    pdftitle={Low-temperature behavior of the multicomponent Widom-Rowlison model on finite square lattices},%
    pdfauthor={\myName, \myUni},%
}

\usepackage{tabularx} 
\usepackage{caption}
\usepackage{subfig}  
\usepackage{placeins}
\usepackage{psfrag}
\usepackage{graphicx}
\usepackage{epstopdf}

\newcommand{\ie}{i.e., }

\newcommand{\ed}{\,{\buildrel d \over =}\,}
\newcommand{\cd}{\xrightarrow{d}}
\newcommand{\E}{\mathbb E}
\newcommand{\pr}[1]{\mathbb P \Big ( #1 \Big )}
\newcommand{\prin}[1]{\mathbb P ( #1 )}
\newcommand{\st}{\mathbin{\lvert}}
\newcommand{\N}{\mathbb N}
\newcommand{\R}{\mathbb R}
\newcommand{\Z}{\mathbb Z}
\renewcommand{\b}{\beta}
\newcommand{\e}{\varepsilon}
\newcommand{\h}{\eta}
\renewcommand{\o}{\omega}
\renewcommand{\t}{\tau}
\renewcommand{\l}{\lambda}
\newcommand{\s}{\sigma}

\newcommand{\G}{\Gamma}
\renewcommand{\L}{\Lambda}
\newcommand{\cX}{\mathcal{X}}
\renewcommand{\ss}{\cX^s}
\newcommand{\binf}{\b \to \infty}
\newcommand{\limb}{\lim_{\binf}}
\newcommand{\rmexp}{\mathrm{Exp}}

\newcommand{\LT}{\mathcal{L}}
\newcommand{\bs}{\mathbf{s}}
\newcommand{\bsp}{\mathbf{s}'}
\renewcommand{\aa}{\bs}
\newcommand{\bb}{\bsp}
\newcommand{\xtbb}{\{X^\b_t \}_{t \in \N}}
\newcommand{\tss}{\t^{\bs}_{\bsp}}
\newcommand{\tsx}{\t^{\bs}_{\ss \setminus \{\bs\}}}
\newcommand{\mm}{M} 
\newcommand{\D}{\Delta H}
\newcommand{\tha}{\tau^\s_{A}}
\newcommand{\xx}{\mathbf{x}}
\newcommand{\yy}{\mathbf{y}}
\newcommand{\ww}{\mathbf{w}}
\newcommand{\zz}{\mathbf{z}}
\newcommand{\bbb}{\mathbf{b}}

\newtheorem{thm}{Theorem}[section]

\newtheorem{lem}[thm]{Lemma}
\newtheorem{prop}[thm]{Proposition}
\theoremstyle{definition}

\title{Low-temperature behavior of the multicomponent\\Widom-Rowlison model on finite square lattices} 
\date{\today}
\author{Alessandro~Zocca\thanks{Centrum Wiskunde \& Informatica, Amsterdam. Email: \texttt{zocca@cwi.nl}}}

\begin{document}

\maketitle

\begin{abstract}
\noindent We consider the multicomponent Widom-Rowlison with Metropolis dynamics, which describes the evolution of a particle system where $M$ different types of particles interact subject to certain hard-core constraints. Focusing on the scenario where the spatial structure is modeled by finite square lattices, we study the asymptotic behavior of this interacting particle system in the low-temperature regime, analyzing the tunneling times between its $M$ maximum-occupancy configurations, and the mixing time of the corresponding Markov chain. In particular, we develop a novel combinatorial method that, exploiting geometrical properties of the Widom-Rowlinson configurations on finite square lattices, leads to the identification of the timescale at which transitions between maximum-occupancy configurations occur and shows how this depends on the chosen boundary conditions and the square lattice dimensions.
\end{abstract}

\section{Introduction}
\label{sec1}
The Widom-Rowlinson model was originally introduced in the chemistry literature by~\cite{WR70} as a continuum model of particles living in $\mathbb R^d$ to study the vapor-liquid phase transition. The discrete-space variant we are interested in was first studied in~\cite{LG71}, where the authors model a lattice gas where two types of particles interact. The Widom-Rowlinson model has in fact two equivalent standard formulations - one as a binary gas and the other as a single-species model of a dense (liquid) phase in contact with a rarefied (gas) phase. In the binary gas formulation, the only interaction is a hard-core exclusion between the two species of particles - call them A and B. If the B particles are invisible, the resulting system of A particles yields the model for vapor-liquid phase transitions introduced and discussed by Widom and Rowlinson~\cite{WR70}. 

The Widom-Rowlinson model has been proven to exhibit a phase transition both on $\Z^d$~\cite{LG71} and on the continuum~\cite{CCK95,R71}. On general graphs, however, its behavior is not always monotonic, as proved in~\cite{BHW99}. 

Many variants of the Widom-Rowlison model have been studied in the literature.
First of all, the generalization to the case where there is only strong repulsion (and no hard-core interaction) between unlike particles has been investigated both on $\Z^d$~\cite{LG71} and on the continuum~\cite{LL72}. The existence of a phase transition has been established also in the case of non-equal fugacities in~\cite{BricmontKurodaLebowitz1984}.
More recently, the Windom-Rowlison model with asymmetric hard-core exclusion diameters has been studied, both on the lattice~\cite{MSSZ14} and on the continuum~\cite{Mazel2015}.
The Widom-Rowlinson model and its properties have also been studied on different discrete space structures, such as the Bethe lattice~\cite{LMNS95,WW70}, other $d$-dimensional lattices~\cite{H02}, and $d$-regular graphs~\cite{CohenPerkinsTetali2016}.

The generalization to the scenario with a finite number of $\mm$ particle types which we consider in this paper has been introduced by Runnels and Lebowitz~\cite{RL74} under the name \textit{multicomponent Widom-Rowlison model}. In the discrete-space version of this model there is nearest-neighbor hard-core exclusion between particles of different types. The multicomponent Widom-Rowlison model has been studied also in the continuum and with many of the variants described above for the case $M=2$, see e.g.~\cite{GeorgiiZagrebnov2001,LL72,LMNS95,NL97,RL74}.

In this paper we focus on the \textit{dynamics} of the multicomponent Widom-Rowlison model on a \textit{finite graph} $G$. Assuming that the $\mm$ particle types have all the same \textit{fugacity} $\l \geq 1$, the evolution over time of this interacting particle system is described by a reversible single-site update Metropolis Markov chain parametrized by the fugacity. More precisely, at every step a site of $G$ and a particle type are selected uniformly at random; if such a site is unoccupied, then a particle of the chosen type is placed there with probability $1$ if and only if all the neighboring sites are not occupied by particles of different types; if instead the selected site is occupied, the particle is removed with probability $1/\l$. This particle dynamics can be equivalently described in terms of the \textit{inverse temperature} $\b =\log \l$ so that the Markov chain under consideration rewrites as a \textit{Freidlin-Wentzell Markov chain} $\xtbb$ with stationary distribution the \textit{Gibbs measure}
\begin{equation}
\label{eq:gibbs}
	\mu_\b(\s)=\frac{1}{Z_\b(G)} e^{-\b H(\s)}, \quad \s \in \cX(G),
\end{equation}
where $\cX(G)$ is the collection of admissible Widom-Rowlinson configurations on $G$, the Hamiltonian $H: \cX(G) \to \R$ assigns to each admissible configuration an energy proportional to the number of particles it has (regardless of their type), \ie $H(\s)=-(N_1(\s)+\dots+N_M(\s))$, and $Z_\b(\L)$ is the normalizing partition function. 

We are interested in particular in the low-temperature regime for the Widom-Rowlison model in which the Gibbs measure~\eqref{eq:gibbs} favors configurations with a large number of particles. In particular, if the underlying graph is connected, there are exactly $M$ maximum-occupancy configurations, in which all the sites of $G$ are occupied by particles of the same type. For large value of the inverse temperature $\b$, these particle configurations become very ``rigid'', in the sense that it takes a long time for the Markov chain $\xtbb$ to move from one to another, as these transitions involve the occurrence of rare events. Indeed, along any such a transition, the interacting particle system must visit configurations with multiple particle types, which become highly unlikely as $\binf$ in view of~\eqref{eq:gibbs}, since the total number of particles in these mixed-activity configurations is smaller due to the fact that a layer of unoccupied sites that must separate cluster of particles of different types.

In order to describe the low-temperature behavior of the Widom-Rowlison model, the main objects of interest in this paper are the \textit{tunneling times}, \ie the first hitting times between maximum-occupancy configurations, and the \textit{mixing time} of the Markov chain $\xtbb$. 

In the present paper we focus in particular on the case where the graph $G$ is a \textit{finite square lattice}, which we denote by $\L$. By studying the geometrical features of the Widom-Rowlison configurations on this class of graphs, we describe the most likely way for the transitions between maximum-occupancy configurations to occur by identifying the the minimum energy barrier $\G(\L)>0$ along all the paths connecting them in the corresponding \textit{energy landscape}. Our analysis shows in particular how this minimum energy barrier $\G(\L)$ depends on the square lattice $\L$ dimensions, as well as on the chosen boundary conditions.

This analysis of the structural properties of the energy landscape in combination with the general framework for Metropolis Markov chains developed in~\cite{NZB16} leads to our main result, which is presented in the next section. In particular, we characterize the asymptotic behavior for the tunneling times between maximum-occupancy configurations giving sharp bounds in probability and proving that the order of magnitude of their expected values as well as the mixing time are equal to $\G(\L)$ on a logarithmic scale. Furthermore, we prove that the asymptotic exponentiality of the tunneling times scaled by their means as $\binf$.

The low-temperature behavior of another interacting particle system, the \textit{hard-core model}, has already been studied using similar techniques on finite square lattices~\cite{NZB16} and triangular lattices~\cite{Zocca2017}. This is not surprising, as these two models are intimately related, as shown in~\cite[Section 5]{BHW99}.

\section{Model description and main results}
\label{sec2}
Consider a finite undirected graph $G=(V,E)$, which models the spatial structure of the finite volume where $M$ types of particles dynamically interact subject to certain hard-core constraints. The $N$ vertices of the graph $G$ represent all possible sites where particles can reside. Particles can be of $\mm \geq 2$ types, and there is nearest-neighbor hard-core exclusion between unlike particles. In other words, the hard-core constraints are modeled by the set $E$ of edges connecting the pairs of sites that cannot be occupied simultaneously by particles of different types.  We associate a variable $\s(v) \in \{0,1,\dots,\mm\}$ to each site $v \in V$, indicating the absence ($\s(v)=0$) or the presence of a particle of type $m$ ($\s(v)=m$). 

A \textit{Widom-Rowlison (WR) configuration} on $G$ is a particle configuration $\s \in \{0,1, \dots, \mm\}^N$ that does not violate the hard-core constraints between unlike particles in neighboring sites; we denote their collection by
\begin{equation}
\label{eq:wrconstraints}
	\cX:=\{ \s \in \{0,1,\dots,\mm\}^N ~:~ \s(v)\s(w) =0 \, \text{ or } \, \s(v)=\s(w) \quad \forall (v,w) \in E \}.
\end{equation} 
In the special case where there only $\mm=2$ particle types (say A and B), there is a more compact representation for WR configurations: By associating a variable $\s(v) \in \{-1,0,1\}$ to each site $v \in V$, indicating the absence ($0$) or the presence of a particle of type A ($1$) or type B ($-1$) in that site, the collection of WR configurations on $G$ rewrite as
$
	\cX =\{ \s \in \{-1,0,1\}^N ~:~ \s(v)\s(w) \neq -1 \quad \forall (v,w) \in E \}.
$

The evolution of this interacting particle system is described by the Metropolis dynamics. More specifically, we consider the single-site update Markov chain $\smash{\xtbb}$ parametrized by a positive parameter $\b$, representing the \textit{inverse temperature}, with Metropolis transition probabilities
\[
	P_\b(\s,\s'):=
	\begin{cases}
		q(\s,\s') e^{-\b [H(\s')-H(\s)]^+}, 	& \text{ if } \s \neq \s',\\
		1-\sum_{\h \neq \s} P_\b(\s,\h), 	& \text{ if } \s=\s',
	\end{cases}
\]
corresponding to the \textit{energy landscape} $(\cX, H,q)$ where the state space $\cX$ is given in~\eqref{eq:wrconstraints} and the \textit{energy} $H$ and \textit{connectivity function} $q$ are chosen as follows. 

The energy function or Hamiltonian $H: \cX \to \R$ counts the number of particles, regardless of their type:
\begin{equation}
\label{eq:energyfunction_WR}
	H(\s) := - \sum_{v \in V} \mathds{1}_{\{\s(v) \neq 0\}}.
\end{equation}
The connectivity function $q: \cX \times \cX \to [0,1]$ allows only single-site updates and prescribes that a particle occupying a certain site cannot be replaced by a particle of the other type in a single step, \ie we define
\begin{equation}
\label{eq:connectivityfunction_WR}
	q(\s,\s'):=
	\begin{cases}
		\frac{1}{\mm N}, 							& \text{if } d(\s,\s')=1,\\
		0, 													& \text{if } d(\s,\s')>1,\\
		1 - \sum_{\h \neq \s} c(\s,\h), 	& \text{if } \s=\s',
	\end{cases}
\end{equation}
where $d: \cX \times \cX \to \N$ is a distance function on $\cX$ defined as
\begin{equation}
\label{eq:distance_WR}
	d(\s,\s'):=\sum_{v \in V} \left ( \mathds{1}_{\{\s(v)\neq \s'(v)\}} +  \mathds{1}_{\{\s(v)\s'(v)\neq 0\}} \right ).
\end{equation}
The dynamics induced by this energy landscape can be described in words as follows. At every step a site $v \in V$ and a type $m$ are selected uniformly at random; if the selected site $v$ is occupied, the particle therein is removed with probability $e^{-\b}$; if instead the selected site $v$ is empty, then a particle of type $m$ is created at $v$ with probability $1$ if and only if none of the neighboring sites is occupied by particles of different types.

The Markov chain $\smash{\xtbb}$ is aperiodic and irreducible on $\cX$, as well as reversible with respect to the \textit{Gibbs measure} associated to the Hamiltonian $H$, \ie
\[
	\mu_\b(\s):=\frac{1}{Z_\b} e^{-\b H(\s)}, \quad \s \in \cX,
\]
where $Z_\b:=\sum_{\s' \in \cX} e^{-\b H(\s')}$ is the normalizing partition function. At low temperature (\ie for large values of $\b$) the Gibbs measure $\mu_\b$ concentrates around the set $\ss$ of global minima of the Hamiltonian $H$ on $\cX$, to which we will henceforth refer as \textit{stable configurations}. 

For every $m=1\dots,\mm$, let $\bs^{(m)}$ be the WR configurations on $G$ in which all the sites are occupied by particles of type $m$, \ie
\[
	\bs^{(m)}(v) := m, \quad \forall \, v \in V.
\]
It is easy to identify these $M$ configurations as the stable configurations of the Widom-Rowlison model on any \textit{connected} graph $G$, as stated by the following lemma.

\begin{lem}[Stable configurations of Widom-Rownlison model]\label{lem:stablestates}
If $G$ is a connected graph, then the stable configurations of the Widom-Rowlison model with $M$ particle types on $G$ are
\[
	\ss = \{\bs^{(1)},\dots,\bs^{(\mm)}\}.
\]
\end{lem}
\begin{proof}
Clearly $H(\bs^{(m)})=-N$ for every $m=1,\dots,\mm$ and $H(\s) \geq -N$ for every $\s \in \cX$. Suppose by contradiction that there exist a WR configuration $\s \neq \bs^{(1)},\dots,\bs^{(\mm)}$ such that $H(\s)=-N$. Trivially $\s$ cannot have any unoccupied site, otherwise $H(\s) >-N$. Furthermore, since in every pair of sites connected by an edge there must reside particles of the same type, the connectedness of $G$ implies that $\s=\bs^{(m)}$ for some $m$.
\end{proof}

In the rest of the paper we focus on the study of the Widom-Rowlison model on \textit{finite square lattices}. More specifically, given two integers $K,L \geq 2$, we take $G$ to be the $K \times L$ square lattice $\L$ either with periodic boundary conditions or with open boundary conditions.
Some examples of WR configurations on square lattices are shown in Figures~\ref{fig:aabbcc} and~\ref{fig:m3mix}, in which empty sites are displayed as white, while occupied sites are displayed using colors, associating one color to each particle type. Furthermore, we will also use particles' color and type interchangeably.

\begin{figure}[!ht]
	\centering
	\subfloat{\includegraphics[scale=0.95]{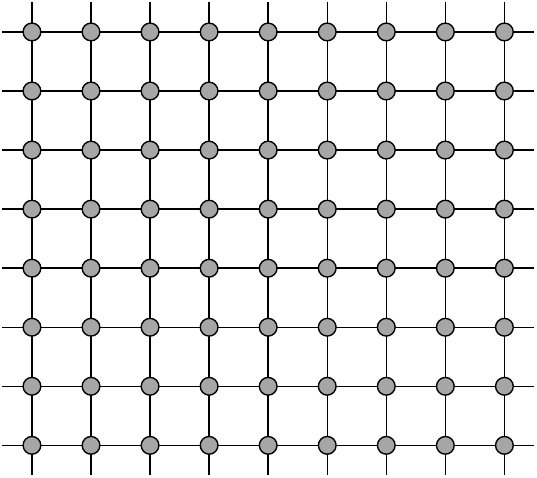}}
	\hspace{0.3cm}
	\subfloat{\includegraphics[scale=0.95]{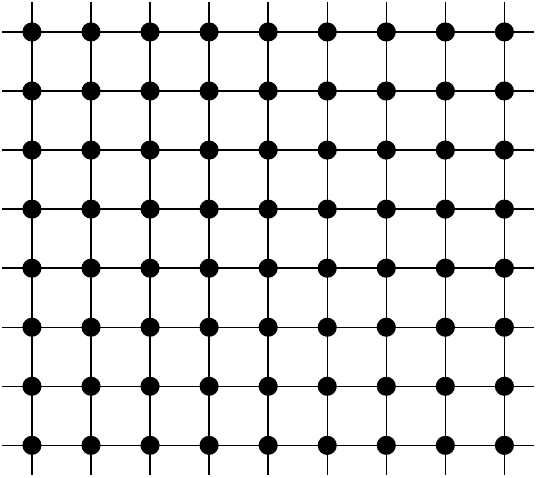}}
	\hspace{0.3cm}
	\subfloat{\includegraphics[scale=0.95]{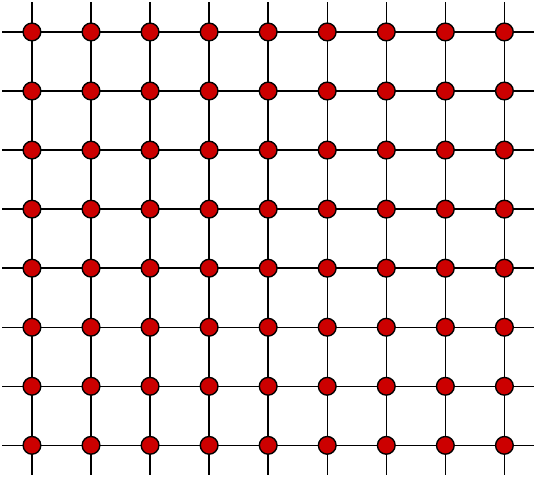}}
	\caption{The three stable configurations $\bs^{(1)},\bs^{(2)},\bs^{(3)}$ of the Widom-Rowlinson model with $\mm=3$ types of particles on the $8\times 9$ square lattice with periodic boundary conditions}
	\label{fig:aabbcc}
\end{figure}
\begin{figure}[!ht]
	\centering
	\subfloat{\includegraphics[scale=0.95]{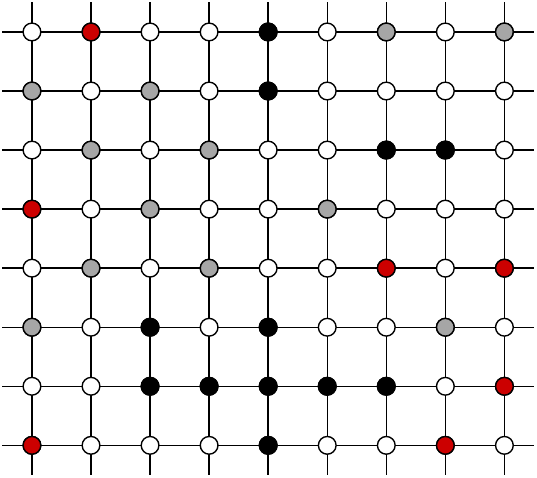}}
	\hspace{0.3cm}
	\subfloat{\includegraphics[scale=0.95]{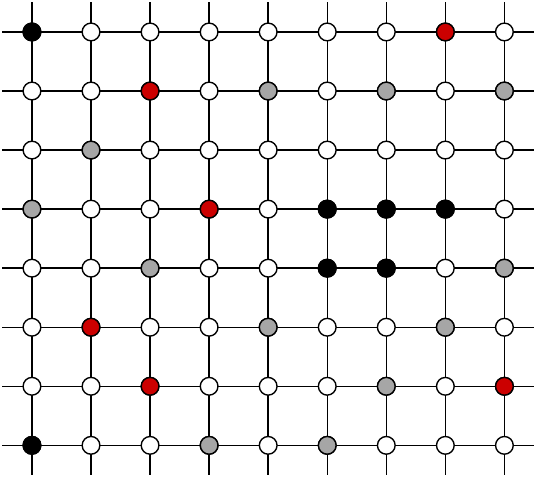}}
	\hspace{0.3cm}
	\subfloat{\includegraphics[scale=0.95]{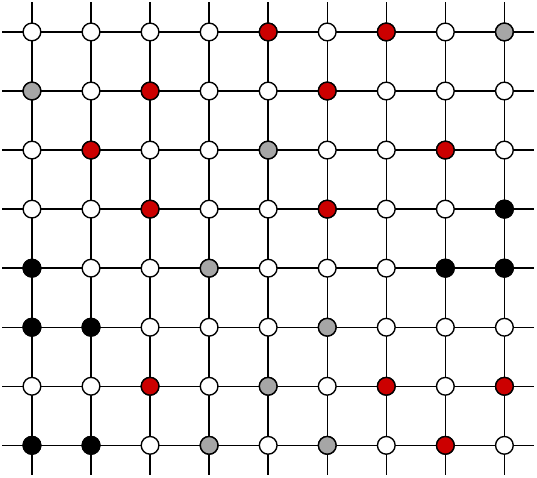}}
	\caption{WR configurations with $\mm=3$ types of particles on the $8\times 9$ square lattice with periodic boundary conditions}
	\label{fig:m3mix}
\end{figure}
\FloatBarrier

Aiming to describe the low-temperature behavior of the Widom-Rowlison model on finite square lattice, we study the Metropolis dynamics $\xtbb$ in this regime and characterize the asymptotic behavior of this interacting particle system in terms of its \textit{tunneling times}, \ie the first hitting times $\tha$ for the Markov chain $\smash{\xtbb}$
\[
	\tha:=\inf \{ t\in \N  : X^{\b}_t \in A \st X^{\b}_0=\s\},
\]
where both the initial and target states are stable configurations, \ie $\s \in \ss$ and $A \subset \ss$. 

For any $\bs,\bsp \in \ss$, we prove the existence of an exponent $\G(\L)>0$ for any finite square lattice $\L$ that gives an asymptotic control in probability of the tunneling times $\smash{\tss}$ and $\smash{\tsx}$ on a logarithmic scale as $\binf$ and characterizes the asymptotic order of magnitude of the expected value. 
The same exponent $\G(\L)$ also describes the timescale at which the Markov chain $\xtbb$ converges to stationarity in the ow-temperature regime: we make this statement rigorous in terms of the \textit{mixing time} of the Markov chain $\xtbb$, defined for every $0 < \e < 1$ as
\[
	t^{\mathrm{mix}}_\b(\e):=\min\{ n \geq 0 ~:~ \max_{\s \in \cX} \| P^n_\b(\s,\cdot) - \mu_\b(\cdot) \|_{\mathrm{TV}} \leq \e \},
\]
where $\| \mu - \mu' \|_{\mathrm{TV}}:=\frac{1}{2} \sum_{\s \in \cX} |\mu(\s)-\mu'(\s)|$ is the total variation distance between any pair of probability distributions $\mu,\mu'$ on $\cX$. Lastly, we show that both the tunneling times $\tss$ and $\tsx$ scaled by their mean converges in distribution to an exponential unit-mean random variable. All these findings are summarized in the following theorem, which is our main result.

\begin{thm}[Low-temperature behavior of the Widom-Rowlinson model on finite square lattices]\label{thm:main}
Consider the Metropolis Markov chain $\xtbb$ corresponding to the Widom-Rowlison model with $\mm$ types of particles on a $K \times L$ finite square lattice $\L$ and define
\begin{equation}
\label{eq:gamma}
	\G(\L) :=
	\begin{cases}
		2K & \text{ if } \L \text{ has periodic boundary conditions, } K = L \text{, and } K+L \geq 6, \\	
		\min \{ 2K, 2L\}+1 & \text{ if } \L \text{ has periodic boundary conditions, } K \neq L \text{, and } K+L \geq 6, \\
		\min \{ K,L \} +1 & \text{ if } \L \text{ has open boundary conditions}.
	\end{cases}
\end{equation}
Then, for any $\bs,\bsp \in \ss$ the following statements hold:
\begin{itemize}
	\item[\textup{(i)}] For every $\e>0$, $\displaystyle \limb \pr{ e^{\b (\G(\L)-\e)} < \tsx \leq \tss < e^{\b (\G(\L)+\e)}} =1;$
	\item[\textup{(ii)}] $\displaystyle \limb \frac{1}{\b} \log \E \tss = \G(\L) = \limb \frac{1}{\b} \log \E \tsx;$
	\item[\textup{(iii)}] $\displaystyle \frac{\tsx}{\E \tsx} \cd \rmexp(1), \quad \mathrm{as} \, \, \binf;$
	\item[\textup{(iv)}] $\displaystyle \frac{\tss}{\E \tss} \cd \rmexp(1), \quad \mathrm{as} \, \, \binf.$
\end{itemize}
Furthermore, 
\begin{itemize}
\item[\textup{(v)}] For any $0 < \e < 1$ the mixing time $t^{\mathrm{mix}}_\b(\e)$ of the Markov chain $\xtbb$ satisfies
\[
	\limb \frac{1}{\b} \log t^{\mathrm{mix}}_\b(\e) = \G(\L).
\]
\end{itemize}
\end{thm}
The proof of this theorem relies on asymptotic results for tunneling and mixing times obtained in~\cite{NZB16} for any Metropolis Markov chains and on the study of structural properties of the energy landscape $(\cX,H,q)$ corresponding to the Widom-Rowlison model on $\L$, which are outlined in the next section.

As illustrated in~\eqref{eq:gamma}, the exponent $\G(\L)$ depends both on the lattice dimensions and on the chosen boundary conditions, and the reason behind this will become clear in Sections~\ref{sec5} and~\ref{sec6}, in which we analyze geometrical features of WR configurations on square lattices and develop the combinatorial approach that lead to the identification of its value. The additional assumption that $K+L\geq 6$ in the case of periodic boundary conditions is necessary to leave out the three special cases $(K,L)=(2,2),(3,2),(2,3)$ in which our combinatorial argument does not work. However, it is easy to prove analogous results for these three special cases, in which the exponent $\G(\L)$ can be easily shown to take respectively the values $3$, $4$, and $4$.

We remark that from our analysis also yields similar results for a square lattice $\L$ with semi-periodic boundary conditions (\ie periodic in one direction and open in the other direction). Indeed, by opportunely combining the results derived in Sections~\ref{sec5} and~\ref{sec6} for the two types of boundary conditions, we can derive that for this class of graphs $\G(\L)= \min \{ K, 2L \}+1$. A similar argument has also been used for the hard-core model on finite square lattice with semi-periodic boundary conditions in~\cite[Section 5.3]{NZB16} and for this reason the details are here omitted.

\section{Analysis of the energy landscape}
\label{sec3}
In this section we derive some structural properties of the energy landscapes of the Widom-Rowlison model on square lattices and show how they can be used to prove our main result. Such properties are summarized in Theorem~\ref{thm:structuralpropertiesWR} below, but we need first to introduce some definitions and useful notation.

It is easy to check that the connectivity matrix $q$ given in~\eqref{eq:connectivityfunction_WR} is irreducible, \ie for any pair of configurations $\s,\s'\in \cX$, $\s \neq \s'$, there exists a finite sequence $\o$ of configurations $\o_1,\dots,\o_n \in \cX$ such that $\o_1=\s$, $\o_n=\s'$ and $q(\o_i,\o_{i+1})>0$ for $i=1,\dots, n-1$ (\ie $\o_i$ and $\o_{i+1}$ differ by an admissible single-site update). We refer to such a sequence as a \textit{path} from $\s$ to $\s'$ and denote it as $\o: \s \to \s'$. The \textit{height} $\Phi_\o$ of a path $\o=(\o_1,\dots,\o_n)$ is the maximum value that the energy takes along $\o$, \ie
\[
	\Phi_\o:= \max_{i=1,\dots,n} H(\o_i).
\]
The \textit{communication height} between any pair of configurations $\s,\s' \in\cX$ is defined as
\[
	\Phi(\s,\s') := \min_{\o : \s\to \s'} \Phi_\o = \min_{\o : \s\to \s'} \max_{i=1,\dots,|\o|} H(\o_i),
\]
and this definition naturally extends to any pair disjoint non-empty subsets $A, B \subset \cX$ as
\[
	\Phi(A,B) := \min_{\s \in A, \, \s' \in B} \Phi(\s,\s').
\]

As stated by the next theorem, we prove that the energy barrier $\Phi(\bs,\bsp) - H(\bs)$ between any pair of stable configurations $\bs,\bsp \in \ss$ is equal to $\G(\L)$ and show that this is the highest energy barrier in the whole energy landscape. 

\begin{thm}[Energy landscape properties for the Widom-Rowlison model on finite square lattices]\label{thm:structuralpropertiesWR} $ $\\
Let $(\cX,H,q)$ be the energy landscape corresponding to the Metropolis dynamics of the Widom-Rowlison model with $\mm$ types of particles on a $K \times L$ square lattice $\L$ and recall the definition of $\G(\L)$ in~\eqref{eq:gamma}. Then, 
\begin{itemize}
	\item[\textup{(i)}] $\Phi(\bs,\bsp) - H(\bs)=\G(\L)$ for every $\bs,\bsp \in \ss$, and
	\item[\textup{(ii)}] $\Phi(\s, \ss) - H(\s) < \G(\L) \quad \forall \, \s \in \cX \setminus \ss$,
\end{itemize}
\end{thm}

The proofs of these statements exploit geometrical features of the WR configurations and therefore will be given separately for square lattices with periodic boundary conditions in Section~\ref{sec5} and with open boundary conditions in Section~\ref{sec6}. 
We remark that the maximum energy barrier $\G(\L)$ of the energy landscape corresponding to the multicomponent Widom-Rowlison model on square lattices does not depend on the number $\mm$ of particle types, which, indeed, has only an entropic effect (the cardinality of $\cX$ obviously grows with the number of particle types).\\

The rest of the section is devoted to the derivation of our main result, Theorem~\ref{thm:main}. The main idea of the proof is to use the structural properties of the energy landscape outlined in Theorem~\ref{thm:structuralpropertiesWR} in combination with the model-independent theory for first hitting times and mixing times developed for a general Metropolis Markov chain in~\cite{NZB16} as a extension of the classical \textit{pathwise approach}~\cite{MNOS04}. Therefore, we first briefly summarize in the next proposition a special case of the results derived in~\cite{NZB16} that is relevant for the tunneling and mixing times.
\newpage
\begin{prop}[Hitting and mixing times asymptotics~\cite{NZB16}] \label{prop:modindep}
Let  $(\cX,H,q)$ be an energy landscape and $\xtbb$ the corresponding Metropolis Markov chain.
\begin{itemize}
\item[\textup{(i)}] Let $\bs \in \ss$ be a stable configuration and define $\G^*:= \max_{\h \in \cX, \, \h \neq \bs} \Phi(\h,\bs) - H(\h)$. Then, the definition of $\G^*$ does not depend on the chosen configuration $\bs \in \ss$ and for any $0 < \e < 1$
\[
	\limb \frac{1}{\b} \log t^{\mathrm{mix}}_\b(\e) = \G^*.
\]
\item[\textup{(ii)}] Assume that a non-empty subset $A \subset \cX$ and a configuration $\s \in \cX \setminus A$ are such that the  identity
\begin{equation}
\label{eq:suffcondPE}
	\Phi(\s,A) - H(\s) = \max_{\h \in \cX \setminus A} \Phi(\h,A) - H(\h)
\end{equation}
holds. Then, setting $\G:= \Phi(\s,A) - H(\s)$, we have that
\[
	\forall \, \e>0 \quad \limb \pr{ e^{\b (\G-\e)} < \tha < e^{\b (\G+\e)}} =1, \qquad \text{ and } \qquad \limb \frac{1}{\b} \log \E \tha = \G.
\]
\item[\textup{(iii)}] Assume that a non-empty subset $A \subset \cX$ and a configuration $\s \in \cX \setminus A$ are such that the inequality
\begin{equation}
\label{eq:suffcondAE}
	\Phi(\s,A) - H(\s) > \max_{\h \in \cX \setminus A, \, \h \neq \s} \Phi(\h, A \cup \{\s\}) - H(\h)
\end{equation}
holds. Then,
\[
	\frac{\tha}{ \E \tha} \cd \rmexp(1), \quad \mathrm{ as } \, \, \binf.
\]
More precisely, there exist two functions $k_1(\b)$ and $k_2(\b)$ with $\limb k_1(\b)=0=\limb k_2(\b)$ such that for any $s>0$
\[
	\Big | \pr{\frac{ \tha }{\E \tha} > s} - e^{-s} \Big | \leq k_1(\b) e^{-(1-k_2(\b))s}.
\]
\end{itemize}
\end{prop}

Condition~\eqref{eq:suffcondPE} says that the energy barrier separating the initial configuration $\s$ from the target subset $A$ is maximum over the entire energy landscape. The authors in~\cite{NZB16} refer to this condition as ``absence of deep cycles'', since it means that all other ``valleys'' of the energy landscape (or more formally \textit{cycles}, see definition in~\cite{MNOS04}) are not deeper than the one where the Markov chain starts. On the other hand, condition~\eqref{eq:suffcondAE} guarantees that for every $\h \in \cX$ the Markov chain $\xtbb$ started in $\h$ most likely reaches either $\s$ or the subset $A$ on a time scale strictly smaller than that at which the transition from $\s$ to $A$ occurs in the limit $\binf$. We note that both conditions~\eqref{eq:suffcondPE} and~\eqref{eq:suffcondAE} are sufficient but not necessary for the corresponding statements (ii) and (iii) to hold; we refer to~\cite{NZB16} for a more elaborate discussion. 

We have now all the ingredients to present the proof of our main result, Theorem~\ref{thm:main}, leaving out however the proof of statement (iv), that is presented later, as it requires some additional work.

\begin{proof}[Proof of Theorem~\ref{thm:main}(i)-(iii) and (v)]
Let $\bs,\bsp \in \ss$ be any two stable configurations. Theorem~\ref{thm:structuralpropertiesWR}(ii) yields
\[
	 \max_{\s \not\in \ss} \Phi(\s,\ss \setminus \{\bs\}) - H(\s) < \G(\L) = \Phi(\bs, \ss\setminus \{\bs\}) - H(\bs),
\]
and, therefore, condition~\eqref{eq:suffcondPE} trivially holds for the pair $(\bs, \ss\setminus \{\bs\})$. Furthermore, taking into account also Theorem~\ref{thm:structuralpropertiesWR}(i), we get
\[
	 \max_{\s \neq \bs} \Phi(\s,\bsp) - H(\s) = \max \left \{  \max_{\s \not\in \ss} \Phi(\s,\bsp) - H(\s),  \max_{\s \in \ss \setminus \{\bs\}} \Phi(\s,\bsp) - H(\s)\right  \}= \G(\L) = \Phi(\bs, \bsp) - H(\bs),
\]
showing that condition~\eqref{eq:suffcondPE} holds also for the pair $(\bs, \bsp)$. Proposition~\ref{prop:modindep}(ii) then yields the statements (i) and (ii) of Theorem~\ref{thm:main} for both tunneling times $\tss$ and $\tsx$.

We now turn to the proof of Theorem~\ref{thm:main}(iii): The asymptotic exponentiality of the rescaled tunneling time $\tsx / \E \tsx$ immediately follows from Proposition~\ref{prop:modindep}(iii) in view of the fact that the inequality
\[
	\Phi(\bs,\ss \setminus \{\bs\})-H(\bs) = \G(\L) > \max_{\s \not\in \ss} \Phi(\s, \ss)-H(\s)
\]
holds by virtue of Theorem~\ref{thm:structuralpropertiesWR}(i) and (ii). Theorem~\ref{thm:structuralpropertiesWR}(i) also implies that the exponent $\G^*$ appearing in Proposition~\ref{prop:modindep}(i) is equal to $\G(\L)$ and this concludes the proof of statement (v).
\end{proof}

We now turn to the proof of the asymptotic exponentiality of $\tss / \E \tss$ in the case $\mm>2$. Indeed, if there are only $\mm=2$ types of particles, Theorem~\ref{thm:main}(iv) trivially holds being equivalent to Theorem~\ref{thm:main}(iii), in view of the fact that $\smash{\tss \ed \tsx}$. However, when $\mm>2$, the asymptotic exponentiality of the scaled tunneling time $\tss / \E \tss$ does not follow immediately from the model-independent results for first hitting times outlined in Proposition~\ref{prop:modindep}. Indeed, in this scenario there more than two stable configurations, since $|\ss| = \mm >2$, and any pair $\aa,\bb \in \ss$ does not satisfy condition~\eqref{eq:suffcondAE} due to the presence of deep cycles (corresponding to the other stable configurations in $\ss \setminus \{\aa,\bb\}$). Indeed, by virtue of Theorem~\ref{thm:structuralpropertiesWR}(i)
\[
	\smash{\G(\L) = \Phi(\bs,\bsp)-H(\bs) \not > \max_{\s \neq \bs, \bsp} \Phi(\s, \{ \bs,\bsp\})-H(\s) = \G(\L).}
\]
The asymptotic exponentiality of the scaled tunneling time $\t^{\aa}_{\bb} / \E \t^{\aa}_{\bb}$ in the limit $\binf$ is proved using a representation of the tunneling time $\tss$ as a geometric random sum of i.i.d.~random variables, stated in the following proposition, which exploits the intrinsic symmetry of the energy landscape $(\cX, H, q)$ corresponding to the Widom-Rownlison model on square lattices. A very similar approach has been also used in~\cite{Zocca2017} to study the tunneling time between any two stable configurations of the hard-core model on finite triangular lattices.

\begin{prop}[Tunneling time $\tss$ as geometric random sum] \label{prop:stochasticrepresentation}
For every $\bs \in \cX$ and any $\b >0$, the following properties hold for the Metropolis Markov chain $\xtbb$ corresponding to the Widom-Rowlison model with $\mm>2$ particle types on a square lattice $\L$:
\begin{itemize}
	\item[\textup{(a)}] The random variable $X^{\b}_{\tsx}$ has a uniform distribution over $\ss \setminus \{\bs\}$;
	\item[\textup{(b)}] The random variable $\tsx$ does not depend on $\bs$;
	\item[\textup{(c)}] The random variables $\tsx$ and $X^{\b}_{\tsx}$ are independent.
\end{itemize}
Furthermore, let $\{\t^{(i)}\}_{i \in \N}$ be a sequence of i.i.d.~random variables with common distribution $\smash{\tau \ed \tsx}$  and $\mathcal{G}_M$ an independent geometric random variable with success probability equal to $p_M:=(M-1)^{-1}$, \ie
\[
	\prin{\mathcal G_{M} = n} = \left ( 1- \frac{1}{M-1}\right )^{n-1} \frac{1}{M-1}, \quad n \geq 1.
\]
Then, for any pair $\bs,\bsp \in \ss$ the following stochastic representation of the tunneling time $\tss$ holds
\begin{equation}
\label{eq:geometricrepresentation}
	\tss \, \ed \sum_{i=1}^{\mathcal{G}_M} \t^{(i)},
\end{equation}
and, in particular, $\E \tss = \frac{1}{p_M} \cdot \E \tsx = (M-1) \cdot \E \tsx$. 
\end{prop}
\begin{proof}
Consider four particle types $x,y,w,z \in \{1,\dots,M\}$ such that $x \neq w$ and $y \neq z$, and the four corresponding stable configurations $\xx, \yy, \ww, \zz \in \ss$. 
Consider the following automorphism $\xi$ of the state space $\cX$ that associates to each WR configuration $\s$ a new WR configuration $\xi(\s)$ by replacing every particles of type $x$ ($y, w, z$) with a particle of type $y$ ($x, z, w$, respectively). More formally, for any $\s \in \cX$ and every $v \in \L$ we define
\[
	[\xi(\s)](v):=
	\begin{cases}
		x  &\text{ if } \s(v) = y,\\
		y &\text{ if } \s(v) = x,\\
		w &\text{ if } \s(v) = z,\\
		z &\text{ if } \s(v) =w,\\
		\s(v) &\text{ if } \s(v) \neq x,y,w,z.
	\end{cases}
\]
$\xi$ is indeed an automorphism of the state space diagram, seen as a graph with vertex set $\cX$ and such that any pair of states $\s,\s' \in \cX$ is connected by an edge if and only if $d(\s,\s') \leq 1$ (see definition~\eqref{eq:distance_WR}). By construction, the automorphism $\xi$ maps $\xx$ to $\yy$ and $\ww$ to $\zz$ simultaneously, \ie
\begin{equation}
\label{eq:automorphism_xywz}
	\xi(\xx)=\yy, \quad \xi(\ss \setminus \{\xx\})= \ss \setminus \{\yy\}, \quad \text{ and } \quad \xi(\ww)=\zz.
\end{equation}
Assume the Markov chain $\{X^\b_t\}_{t \in \N}$ starts in $\xx$ at time $0$. Let $\{ \widetilde{X}^\b_t \}_{t\in \N}$ be the Markov chain that mimics the moves of $X^\b_t$ via the automorphism $\xi$, \ie set $\widetilde{X}^\b_t:=\xi(X^\b_t)$ for every $t \in \N$. 

For notational compactness, we suppress in this proof the dependence on $\b$ of these two Markov chains. Since $\xi$ is an automorphism, for any pair of WR configurations $\s,\s' \in \cX$, any transition of $\widetilde{X}_t$ from $\h=\xi(\s)$ to $\h'=\xi(\s')$ is feasible and occurs with the same probability as the transition from $\s$ to $\s'$. We have defined in this way a coupling between $\{X_t\}_{t\in \N}$ and $\{\widetilde{X}_t\}_{t \in \N}$, which are two copies of the same Markov chain describing the Metropolis dynamics for the Widom-Rowlinson model on $\L$. In view of~\eqref{eq:automorphism_xywz}, this coupling immediately implies that the Markov chain $\{X_t\}_{t\in \N}$ started at $\xx$ hits a stable configuration in $\ss \setminus \{\xx\}$ precisely when the Markov chain $\{\widetilde{X}_t\}_{t \in \N}$ started at $\yy=\xi(\xx)$ hits a configuration in $\ss \setminus \{\yy\}$. Furthermore, for every $\xx,\yy,\ww,\zz \in \ss$ such that $\xx \neq \ww$ and $\yy \neq \zz$, the following identity holds:
\begin{equation}
\label{eq:jointdistribution}
	\pr{X_{\t^\xx_{\ss \setminus \{\xx\}}} =\ww , \, \t^\xx_{\ss \setminus \{\xx\}} \leq t} = \pr{\xi(X_{\t^\xx_{\ss \setminus \{\xx\}}})=\xi(\ww), \, \t^{\xi(\xx)}_{\xi(\ss \setminus \{\xx\})} \leq t}  = \pr{\widetilde{X}_{\t^\yy_{\ss \setminus \{\yy\}}}=\zz , \, \t^\yy_{\ss \setminus \{\yy\}} \leq t}.
\end{equation}
Taking $\xx=\yy$ and the limit $t \to \infty$ in~\eqref{eq:jointdistribution}, we obtain that for every $\ww,\zz \in \ss \setminus \{\xx\}$
\[
	\pr{X_{\t^\xx_{\ss \setminus \{\xx\}}} =\ww} = \pr{\widetilde{X}_{\t^\xx_{\ss \setminus \{\xx\}}}=\zz}.
\]
Using the fact that $\{X_t\}_{t\in \N}$ and $\{\widetilde{X}_t \}_{t\in \N}$ have the same statistical law, being two copies of the same Metropolis Markov chain, it then follows that the random variable $\smash{X_{\t^\xx_{\ss \setminus \{\xx\}}}}$ has a uniform distribution over $\ss \setminus \{\xx\}$, that is property (a). In particular, for any $\smash{\bs \in \ss \setminus \{\xx\}}$
\begin{equation}
\label{eq:uniformdistribution}
	\pr{X_{\t^\xx_{\ss \setminus \{\xx\}}} =\bs }=\frac{1}{|\ss \setminus \{\xx\}|} = \frac{1}{M-1}.
\end{equation}
By summing over $\ww \in \ss \setminus \{\xx\}$ in~\eqref{eq:jointdistribution}, we have that for every $\zz \in \ss \setminus \{\yy\}$ and every $t \geq 0$
\begin{equation}
\label{eq:partialidentity}
	\pr{\t^\xx_{\ss \setminus \{\xx\}} \leq t} = (M-1) \cdot \pr{X_{\t^\yy_{\ss \setminus \{\yy\}}}=\zz, \, \t^\yy_{\ss \setminus \{\yy\}} \leq t}.
\end{equation}
By summing over $\zz \in \ss \setminus \{\yy\}$ in~\eqref{eq:partialidentity}, we obtain that for every $t \geq 0$
\begin{equation}
\label{eq:equaldistribution}
	\pr{\t^\xx_{\ss \setminus \{\xx\}} \leq t} = \pr{\t^\yy_{\ss \setminus \{\yy\}} \leq t},
\end{equation}
proving property (b). Substituting~\eqref{eq:equaldistribution} into identity~\eqref{eq:partialidentity} and using~\eqref{eq:uniformdistribution}, we deduce that every $\yy,\zz \in \ss$ with $\yy\neq \zz$ and for every $t\geq 0$,
\[
	\pr{X_{\t^\yy_{\ss \setminus \{\yy\}}} = \zz, \, \t^\yy_{\ss \setminus \{\yy\}} \leq t} = \pr{X_{\t^\yy_{\ss \setminus \{\yy\}}}=\zz} \pr{\t^\yy_{\ss \setminus \{\yy\}} \leq t},
\]
that is property (c).

We now proceed with the derivation of the stochastic representation in~\eqref{eq:geometricrepresentation} for the tunneling time $\tss$. Let $\mathcal G_{M}$ be the random variable that counts the number of non-consecutive visits to stable configurations in $\ss \setminus \{\bsp\}$ until the configuration $\bb$ is hit, counting as first visit the configuration $\bs$ where the Markov chain starts. Non-consecutive visits here means that we count as actual visit to a stable configuration only the first one after the last visit to a different stable configuration. Property (b) implies that the random time between these non-consecutive visits does not depend on the last visited stable configuration. In view of property (a), the random variable $\mathcal G_{M}$ is geometrically distributed with success probability equal to $p_M:=(M-1)^{-1}$, \ie
\[
	\prin{\mathcal G_{M} = k} = \left ( 1- \frac{1}{M-1}\right )^{k-1} \cdot \frac{1}{M-1}, \quad k \geq 1.
\]
In particular, note that $\mathcal G_{M}$ depends only on $\mm$ and not on the inverse temperature $\b$. The amount of time $\tsx$ it takes for the Metropolis Markov chain started in $\bs \in \ss$ to hit any stable configuration in $\ss \setminus \{\bs\}$ does not depend on $\bs$, by virtue of property (b). In view of these considerations and using the independence property (c), we deduce~\eqref{eq:geometricrepresentation}. Lastly, the identity $\smash{\E \tss = (M-1) \cdot \E \tsx}$ immediately follows from Wald's identity, since both random variables $\mathcal G_{M}$ and $\smash{\tsx}$ have finite expectation and $\E \mathcal G_{M} = p_M^{-1}$. 
\end{proof}

We now conclude the proof of our main result using together the latter proposition and the already proved Theorem~\ref{thm:main}(iii).
\begin{proof}[Proof of Theorem~\ref{thm:main}(iv)]
Denote by $\LT_{\tss}(\cdot)$ and $\LT_{\tsx}(\cdot)$ the Laplace transforms of the hitting times $\tss$ and $\tsx$, respectively. The stochastic representation~\eqref{eq:geometricrepresentation} given in Proposition~\ref{prop:stochasticrepresentation} implies that
\[
	\LT_{\tss} = G_{\mathcal G_{M}}\left(\LT_{\tsx} (t)\right), \quad \forall \, t \geq 0.
\]
where $G_{\mathcal G_{M}}(z) :=\E (z^{\mathcal G_{M}})$, $z \in [0,1]$, is the probability generating function of the random variable $\mathcal G_{M}$. From Theorem~\ref{thm:main}(iii) it follows that for any stable configuration $\bs \in \ss$,
\[
	\LT_{\tsx /\E \tsx} (t) \stackrel{\binf}{\longrightarrow} \LT_Y(t), \quad \forall \, t \geq 0,
\]
where $Y$ is a unit-mean exponential random variable. Using the fact that $\E \tss  = \E \tsx  \cdot \E \mathcal G_{M}$, we obtain
\[
	\LT_{\tss / \E \tss } (t) = G_{\mathcal G_{M}} \left(\LT_{\tsx /\E \tsx}(t / \E \mathcal G_{M})\right) \stackrel{\binf}{\longrightarrow} G_{\mathcal G_{M}} \left(\LT_{Y}(t / \E \mathcal G_{M})\right), \quad \forall \, t \geq 0,
\]
and the continuity theorem for Laplace transforms yields that
\[
	\frac{\tss}{\E \tss} \cd \frac{1}{\E \mathcal G_{M}} \sum_{i=1}^{\mathcal G_{M}} Y^{(i)}, \quad \mathrm{ as } \, \, \binf,
\]
where $\{Y^{(i)}\}_{i \in \N}$ is a sequence of i.i.d.~exponential random variables with unit mean. The conclusion then follows by noticing that the geometric sum of i.i.d.~exponential random variables scaled by its mean is also exponentially distributed with unit mean. 
\end{proof}

\section{Geometrical features of WR configurations on square lattices}
\label{sec4}
In this section we derive some geometrical properties of WR configurations on finite square lattices, which are extensively used in Sections~\ref{sec5} and~\ref{sec6} to analyze the energy landscape corresponding to the Widom-Rowlinson model on the same graphs.

We first introduce some useful notation and results that are valid for every $K \times L$ square lattice $\L$ with $K,L \geq 2$, regardless of the imposed boundary conditions.
Denote by $c_j$, $j=0,\dots,L-1$, the $j$-th column of $\L$, \ie the collection of sites whose horizontal coordinate is equal to $j$, and by $r_i$, $i=0,\dots,K-1$, the $i$-th row of $\L$, \ie the collection of sites whose vertical coordinate is equal to $i$, see Figure~\ref{fig:stripesgridch7}. When we do not need to specify a precise row or column, we denote a generic column by $c$ and a generic row by $r$. We assign the coordinates $(j,i)$ to the vertex $v$ that lies at the intersection of column $c_j$ and row $r_i$. In addition, define the \textit{$i$-th horizontal stripe}, with $i=1,\dots,\lfloor K/2 \rfloor$, as 
\[
	S_i:=r_{2i-2} \cup r_{2i-1},
\]
and the \textit{$j$-th vertical stripe}, with $j=1,\dots,\lfloor L/2 \rfloor$ as
\[
	C_j:=c_{2j-2} \cup c_{2j-1},
\]
see Figure~\ref{fig:stripesgridch7} for an illustration.
\begin{figure}[!h]
	\centering
	\includegraphics[scale=0.95]{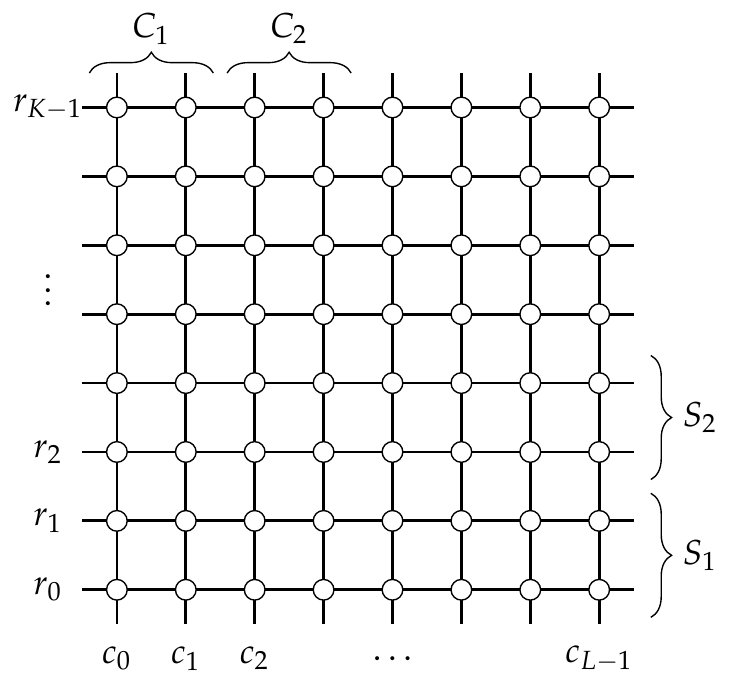}
	\caption{Illustration of row, column and stripe notation on square lattice $\L$ with periodic boundary conditions}
	\label{fig:stripesgridch7}
\end{figure}

\noindent Given an WR configuration $\s \in \cX$, we define $\D(\s)$ as the number of the empty sites that $\s$ has, namely
\begin{equation}
\label{eq:ineff}
	\D(\s) := \sum_{v \in \L} \mathds{1}_{\{\s(v) = 0\}} \stackrel{\eqref{eq:energyfunction_WR}}{=} N - \sum_{v \in \L} \mathds{1}_{\{\s(v)\neq 0\}}.
\end{equation}
Note that $\D(\s)$ can equivalently be seen as the \textit{energy difference} between $\s$ and any stable configuration $\aa \in \ss$, since $\D(\s) = H(\s)-H(\aa)$. We further define the energy difference of a configuration $\s \in \cX$ in row $r$ and in column $c$ respectively by 
\begin{equation}
\label{def:Ur}
	\D^{r}(\s):= L - \sum_{v \in r} \mathds{1}_{\{\s(v)\neq 0\}} = L - \sum_{v \in r} |\s(v)|, \quad \text{ and } \quad \D^{c}(\s):=  K - \sum_{v \in c} \mathds{1}_{\{\s(v)\neq 0\}} = K - \sum_{v \in c} |\s(v)|.
\end{equation}
Clearly the energy difference~\eqref{eq:ineff} of a configuration $\s$ can be written as the sum of the energy differences in every row (or in every column), \ie
\[
	\D(\s) = \sum_{i=0}^{K-1} \D^{r_i}(\s) = \sum_{j=0}^{L-1} \D^{c_j}(\s).
\]
We say that a WR configuration $\s\in \cX$ displays:
\begin{itemize}
	\item A \textit{vertical (horizontal) $m$-bridge} in column (row) if all sites of that column (row) are occupied by particles of type $m$;
	\item A \textit{vertical (horizontal) quasi $m$-bridge} in column (row) if all sites but one of that column (row) are occupied by particles of type $m$ and the remaining site is unoccupied;
	\item A \textit{cross} if it has both a vertical bridge and a horizontal bridge;
	\item A \textit{quasi-cross} if it has both a vertical quasi-bridge and a horizontal quasi-bridge.
\end{itemize}

We will interchangeably refer to bridges, quasi-bridges, and crosses using either the particle type that characterizes them or the color that identifies that particle type. Furthermore, we simply speak of bridges, quasi-bridges, and crosses if their color is not relevant. Some examples are given in Figures~\ref{fig:bridgesandcrossesWR}.

\begin{figure}[!ht]
	\centering
	\subfloat[A black horizontal bridge]{\includegraphics[scale=0.65]{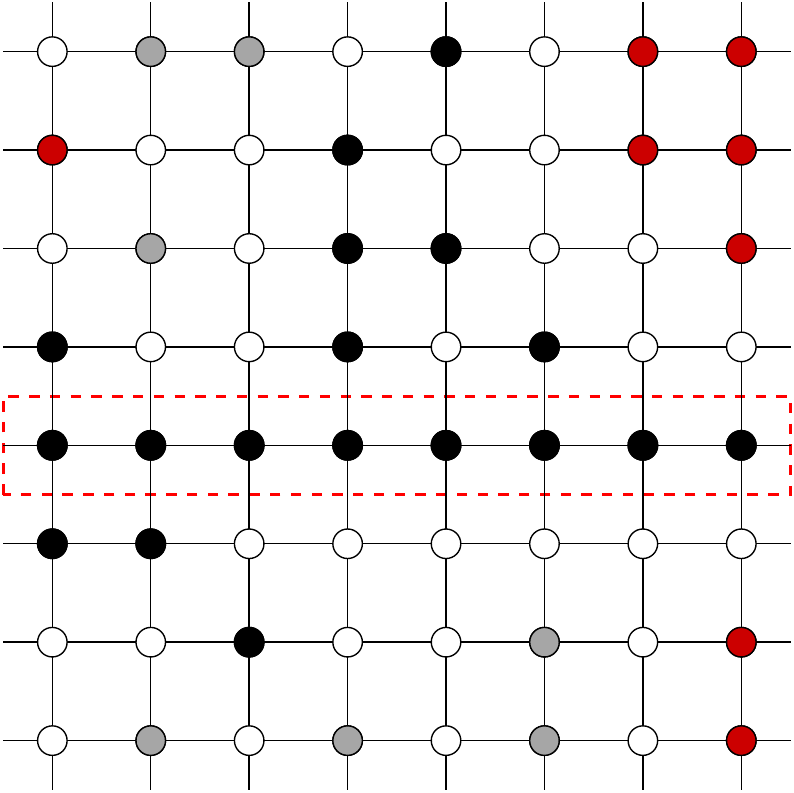}}
	\hspace{0.5cm}
	\subfloat[Two vertical bridges]{\includegraphics[scale=0.65]{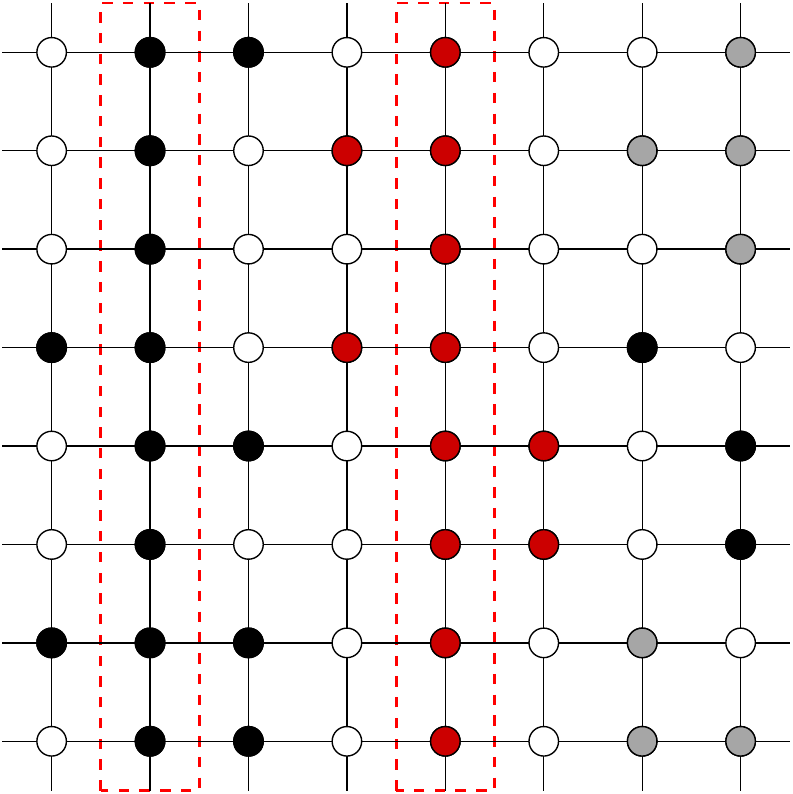}}
	\hspace{0.5cm}
	\subfloat[A black cross]{\includegraphics[scale=0.65]{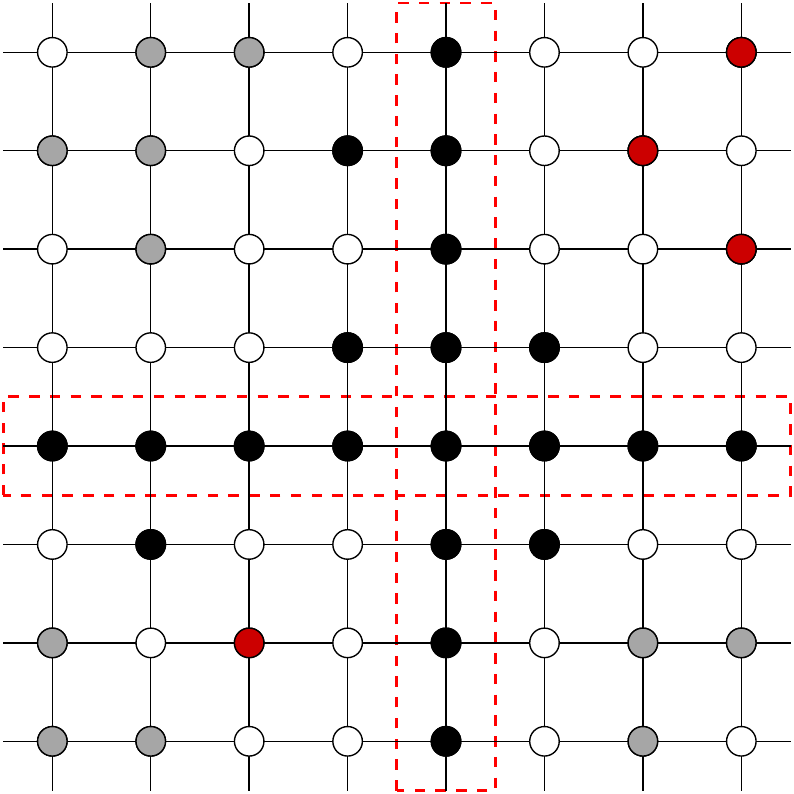}}
	\\
	\subfloat[A vertical gray quasi-bridge]{\includegraphics[scale=0.65]{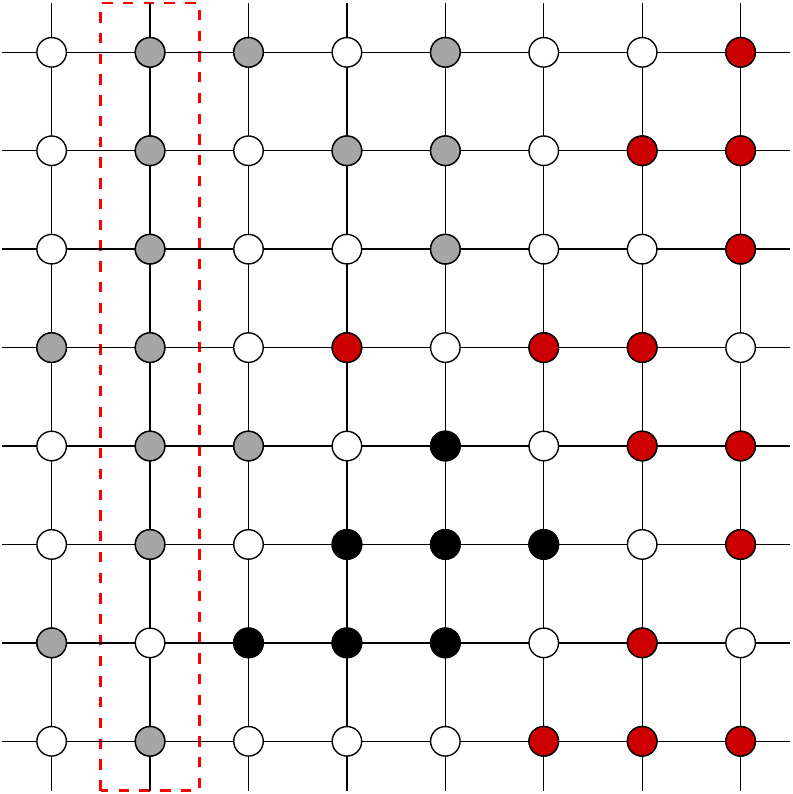}}
	\hspace{0.5cm}
	\subfloat[Two horizontal gray quasi-bridges]{\includegraphics[scale=0.65]{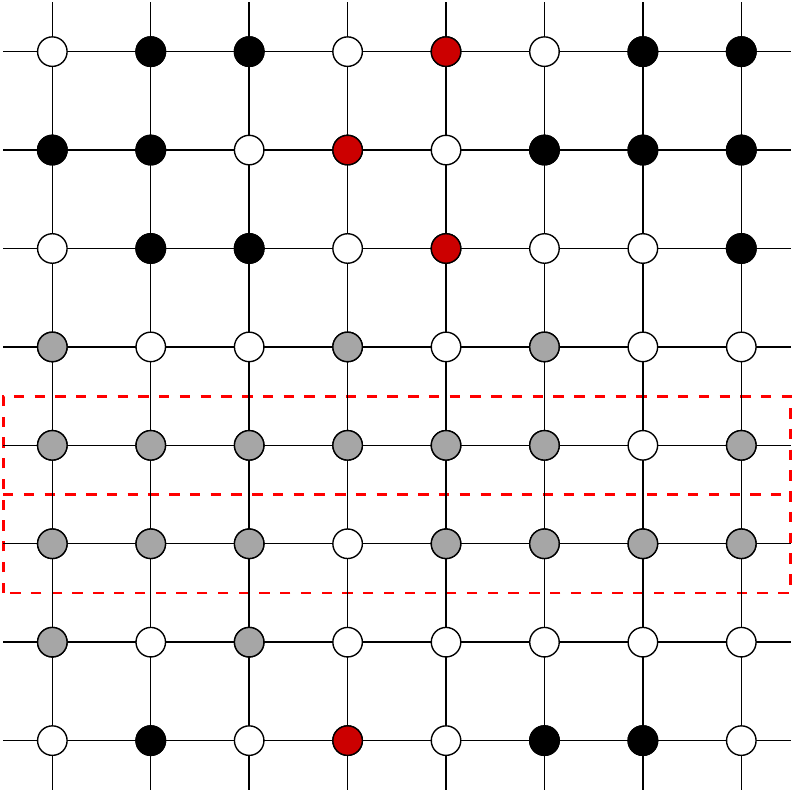}}
	\hspace{0.5cm}
	\subfloat[A gray quasi-cross]{\includegraphics[scale=0.65]{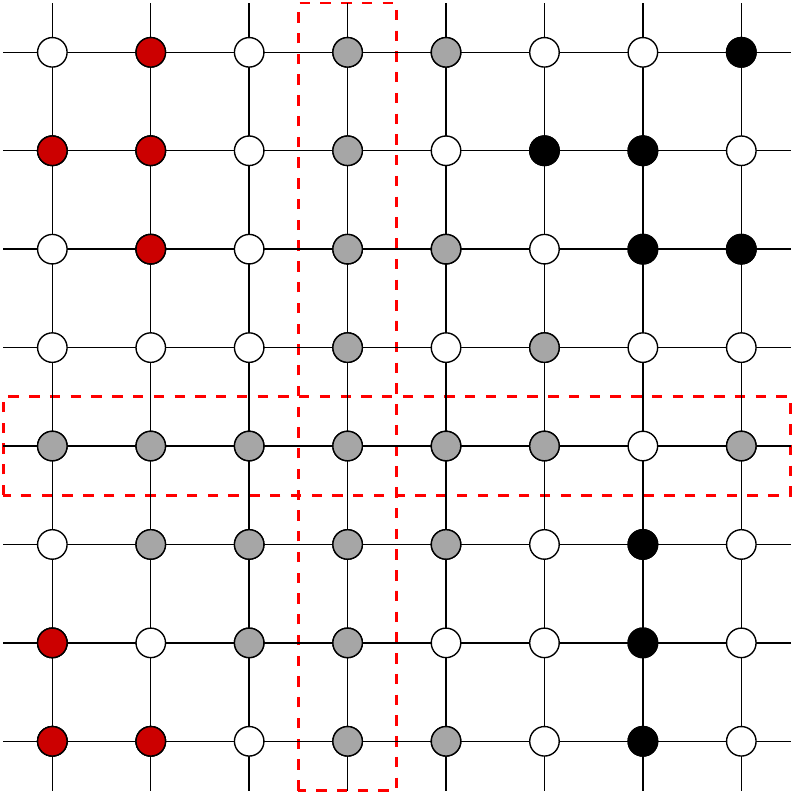}}
	\caption{WR configurations with $\mm=3$ types of particles displaying bridges and quasi-bridges on the $8 \times 8$ square lattice with periodic boundary conditions}
	\label{fig:bridgesandcrossesWR}
\end{figure}
\FloatBarrier

\begin{lem}[Geometrical features of bridges for WR configurations]\label{lem:impossible}
The following properties hold for a WR configuration on a square lattice $\L$:
\begin{itemize}
	\item[\textup{(i)}] Two bridges of different colors cannot be perpendicular to each other;
	\item[\textup{(ii)}] A bridge and quasi-bridge of different colors cannot be perpendicular to each other;
	\item[\textup{(iii)}] Two quasi-bridges of different colors and perpendicular to each other must meet in their empty site;
	\item[\textup{(iv)}] A (quasi-)bridge can have another (quasi-)bridge in adjacent row/column only if they are of the same color.
\end{itemize}
\end{lem}

The proof of Lemma~\ref{lem:impossible} is immediate by looking at the hard-core constraints between unlike particles arising at the site in which the two bridges or quasi-bridges meet for statements (i)-(iii) and at the neighboring sites belonging to adjacent rows/columns for statement (iv), see for instance Figure~\ref{fig:bridgesandcrossesWR}.\\

In view of Lemma~\ref{lem:impossible}(i), crosses consist of particles of a single type, but this does not have to be the case for quasi-crosses and motivates the following definition. We call a quasi-cross \textit{monochromatic} if it consists of particles of the same type, and \textit{bichromatic} otherwise, see some examples in Figure~\ref{fig:quasicrosses}. Lemma~\ref{lem:impossible}(iii) implies that the two quasi-bridges of a bichromatic quasi-cross must intersect in their unique empty site.  

\begin{figure}[!h]
	\centering
	\subfloat[A monochromatic (black) quasi-cross]{\includegraphics[scale=0.65]{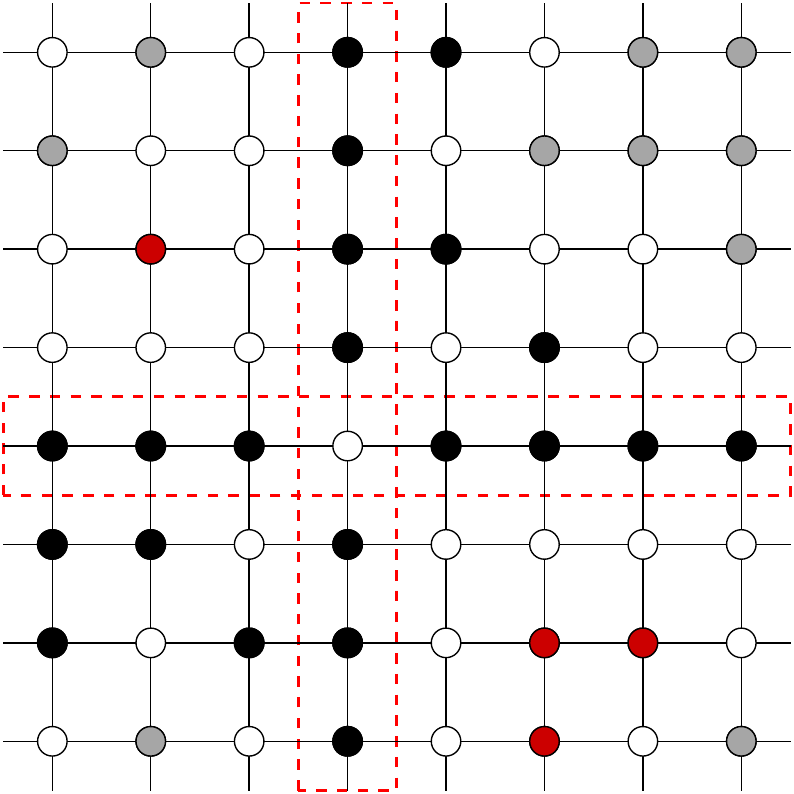}}
	\hspace{0.5cm}
	\subfloat[A monochromatic (gray) quasi-cross]{\includegraphics[scale=0.65]{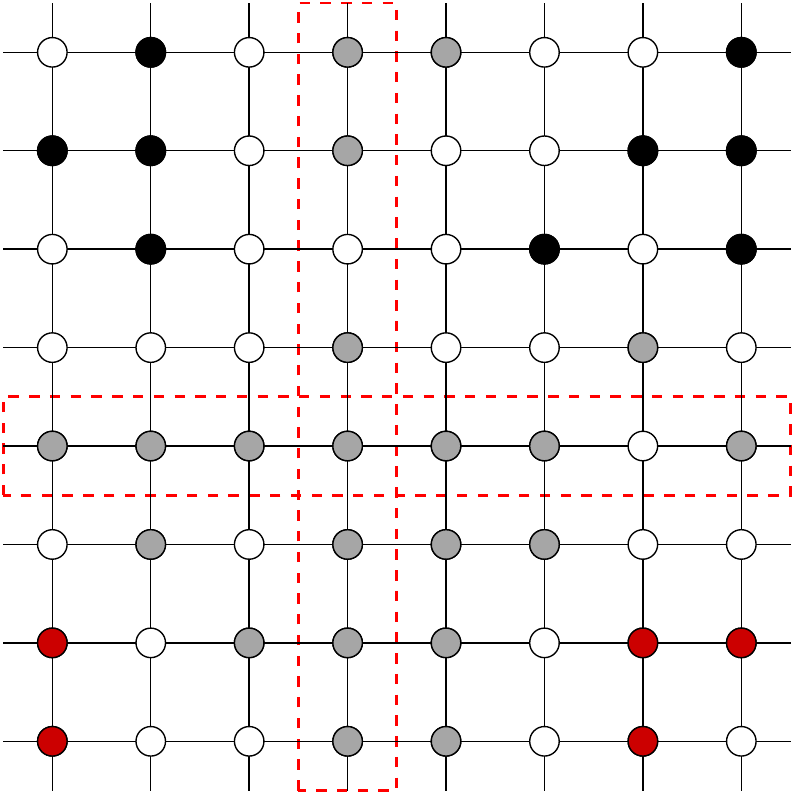}}
	\hspace{0.5cm}
	\subfloat[A bichromatic quasi-cross]{\includegraphics[scale=0.65]{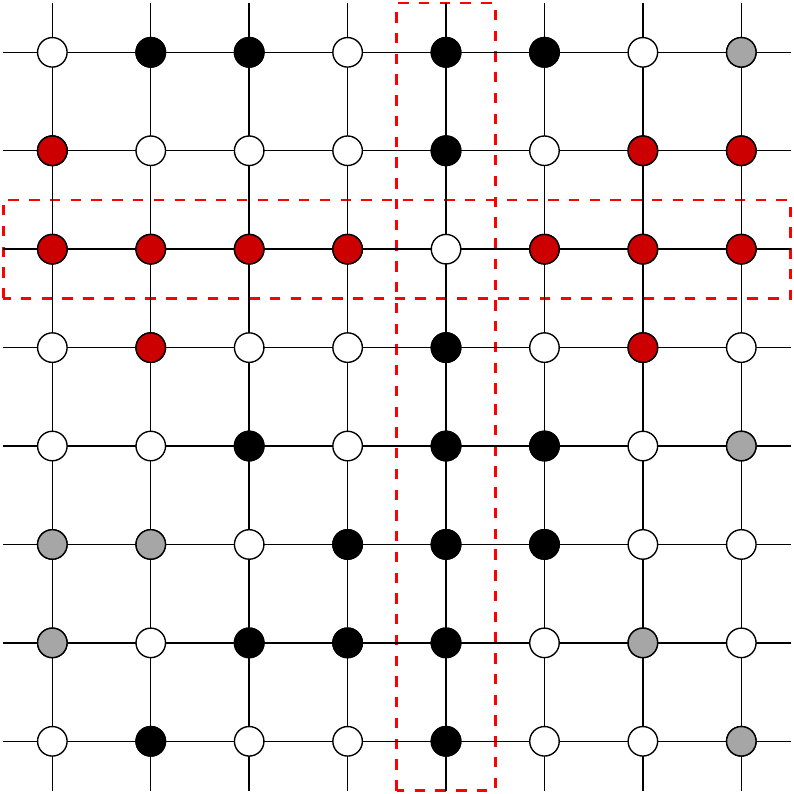}}
	\caption{Examples of WR configurations with $\mm=3$ types of particles displaying quasi-crosses on the $8 \times 8$ square lattice with periodic boundary conditions}
	\label{fig:quasicrosses}
\end{figure}
\FloatBarrier
The next two lemmas show that bridges are the unique particle displacements with zero energy difference in a given row/column, regardless of the chosen boundary conditions. In addition, only in the case of a square lattice $\L$ with periodic boundary conditions, we give an equivalent characterization of quasi-bridges. Lemmas~\ref{lem:bqbi} and~\ref{lem:bqbso_opengrid} are stated and proved only for rows and horizontal (quasi-)bridges, since those for columns and vertical (quasi-)bridges are analogous.

\begin{lem}[Bridges and quasi-bridges characterization on square lattices with periodic boundary conditions]\label{lem:bqbi} $ $\\
Let $\s \in \ss$ be a WR configuration on a square lattice $\L$ with periodic boundary conditions. Then,
\begin{itemize}
	\item[\textup{(i)}] $\D^{r}(\s) =0 \quad \Longleftrightarrow \quad \s$ has a horizontal bridge in row $r$;
	\item[\textup{(ii)}] $\D^{r}(\s) =1 \quad \Longleftrightarrow \quad \s$ has a horizontal quasi-bridge in row $r$.
\end{itemize}
In particular, if $\s$ has no bridges nor quasi-bridges in row $r$, then $\D^{r}(\s) \geq 2$.
\end{lem}
\begin{proof}
The implications (i)$\Leftarrow$ and (ii)$\Leftarrow$ are immediate by definition of (quasi-)bridge and~\eqref{def:Ur}. For the converse implications, is enough to observe that, in order to have particles of different types in the same row $r$, there should be at least two empty sites separating the two (or more) clusters of alike particles from each other, due to the periodic boundary conditions. If $\D^r(\s) \leq 1$, then all the particles residing in row $r$ are of the same type and their number is automatically determined by the value of the energy difference $\D^{r}(\s)$.
\end{proof}

\begin{lem}[Bridges characterization on square lattices with open boundary conditions]\label{lem:bqbso_opengrid} Let $\s \in \ss$ be a WR configuration on a square lattice $\L$ with open boundary conditions. Then, 
\[
	\D^r(\s)=0 \quad \Longleftrightarrow \quad  \s \text{ has a horizontal bridge in row } r.
\]
\end{lem}
\begin{proof}
The implication $\Leftarrow$ follows immediate from the definitions of bridge and energy difference on a row~\eqref{def:Ur}. Assume by contradiction that $\s$ does not have a horizontal bridge in row $r$. If $\s$ has only particles of one type in row $r$ and does not display a bridge there, then the number of particle in row $r$ must be strictly less than $L$, so by definition of energy difference $\D^r(\s) >0$, which is a contradiction. If instead $\s$ has particles of two or more types in row $r$, there has to be at least one empty site separating the clusters created by particles of the same type, and thus $\D^r(\s) \geq 1$, which is again a contradiction.
\end{proof}
In the next two sections these geometrical features of WR configurations on square lattices will be leveraged to prove Theorem~\ref{thm:structuralpropertiesWR}. Even if the underlying ideas are similar, the two possible boundary conditions for the square lattice $\L$ require different combinatorial arguments and, for this reason, we present them separately, first the case of periodic boundary conditions in Section~\ref{sec5} and then that of open boundary conditions in Section~\ref{sec6}.

\newpage
\section{Proofs for square lattices with periodic boundary conditions}
\label{sec5}
This section is devoted to the analysis of the energy landscape of the Widom-Rowlison model on $K \times L$ square lattice $\L$ with periodic boundary conditions, which leads to the proof of the corresponding statements in Theorem~\ref{thm:structuralpropertiesWR}.

The section is organized as follows. First, we present a lower bound for the communication height between any pair $\bs, \bsp$ of stable configurations, see Proposition~\ref{prop:lowertg_WR}. We then introduce a \textit{reduction algorithm}, which is then used in Proposition~\ref{prop:refpathwrt} to build a reference path $\o^*: \bs \to \bsp$. The existence of such a path shows that the lower bound given in Proposition~\ref{prop:lowertg_WR} is sharp and concludes the proof of Theorem~\ref{thm:structuralpropertiesWR}(i). We then use the reduction algorithm to construct a path with prescribed height from every WR configuration $\s \not\in \ss $ to one of the $\mm$ stable configurations proving in this way Theorem~\ref{thm:structuralpropertiesWR}(ii).

\begin{prop}[Lower bound for $\Phi(\aa,\bb)$]\label{prop:lowertg_WR}
Consider the Widom-Rowlison model on the $K \times L$ square lattice $\L$ with periodic boundary conditions with $(K,L) \neq (2,2), (2,3)$. The communication height between any pair of stable configurations $\aa, \bb \in \ss$ in the corresponding energy landscape satisfies
\begin{equation}
\label{eq:lowertg_WR}
	\Phi(\aa,\bb) - H(\aa) \geq
	\begin{cases}
		2K 							& \text{if } K=L,\\
		\min\{2K,2L\}+1 	& \text{if } K\neq L.
	\end{cases}
\end{equation}
\end{prop}
\begin{proof}
Modulo relabeling particle types, we can assume without loss of generality that $\bs=\bs^{(1)}$ and $\bsp=\bs^{(2)}$ and associate the color gray to particle of type $1$ and the color black to those of type $2$. Let $\hat{\xi}: \cX \to \cX$ be the function that maps each configuration $\s \in \cX$ into the configuration $\hat{\xi}(\s)$ with the same empty sites and where all the particles that are not black nor gray are replaced by gray particles, \ie for every $v \in \L$ set
\begin{equation}
\label{eq:projection}
	[\hat{\xi}(\s)](v):=
	\begin{cases}
		\s(v) & \text{ if } \s(v) \in \{0,1,2\},\\
		1 & \text{ if } \s(v) \in \{3, \dots, M\}.
	\end{cases}
\end{equation}
The resulting configuration is clearly a WR configuration, \ie $\hat{\xi}(\s) \in \cX$, and has the same energy as the original one, $H(\s)=H(\hat{\xi}(\s))$, since they have the same number of particles. Hence, every path $\o: \bs \to \bsp$ of length $n$ is mapped by $\hat{\xi}$ to a new path $\o':=\hat{\xi}(\o)$ of the same length from $\bs$ to $\bsp$ and whose energy profile is identical to that of the original path, \ie
\[
	H(\o'_i) = H(\hat{\xi}(\o_i)) = H(\o_i) , \quad \forall \, i =1,\dots,n,
\]
and, in particular, $\Phi_{\o'}=\Phi_\o$. This observation allow us to work for the purpose of this proof as if $\mm=2$.\\

Assuming without loss of generality that $K \leq L$, we need to show that in every path $\o: \aa \to \bb$ there is at least one configuration with energy difference greater than or equal to $\min\{2K+1,2L\}$. Take a path $\o = (\o_1,\dots, \o_n)$ from $\aa$ to $\bb$. Without loss of generality, we may assume that there are no void moves in $\o$, \ie at every step a particle is either added or removed, so that $H(\o_{i+1}) = H(\o_{i}) \pm 1$ for every $1 \leq i \leq n-1$. 

If two WR configurations $\s,\s' \in \cX$ with $d(\s,\s')=1$ (see definition~\eqref{eq:distance_WR}) are such that $\s$ does not display a black bridge in a certain row/column and $\s'$ instead does, then $\s$ must have a quasi-bridge in that row/column. Moreover in this case $H(\s')=H(\s)-1$, since the bridge is created by adding a black particle in the only empty site of that row/column.

Since $\aa$ has no black bridges, while $\bb$ does, at some point along the path $\o$ of length $n$ there must be an index $n^*\leq n$ such that configuration $\o_{n^*}$ that is \textit{the first} to display a black bridge (horizontal or vertical) or a black quasi-cross. Clearly $n^*>2$. We claim that 
\[
	\max \{\D(\o_{n^*-1}), \D(\o_{n^*-2})\} \geq \min\{2K+1,2L\}.
\]
We distinguish the following three cases:
\begin{itemize}
	\item[(a)] $\o_{n^*}$ displays a black vertical bridge only;
	\item[(b)] $\o_{n^*}$ displays a black horizontal bridge only;
	\item[(c)] $\o_{n^*}$ displays a black quasi-cross.
\end{itemize}
Note that in case (c) we do not exclude the possibility that a black quasi-cross is created simultaneously with a black bridge.\\

For case (a), let $c^*$ be the column where $\o_{n^*}$ displays the black vertical bridge. The previous configuration $\o_{n^*-1}$ along the path $\o$ differs from $\o_{n^*}$ in exactly one site, say $v' \in c^*$. By construction, $\o_{n^*-1}(v')=0$ and $\o_{n^*-1}$ has a black vertical quasi-bridge in column $c^*$. In any row the configuration $\o_{n^*-1}$ cannot display
\begin{itemize}
	\item a horizontal black bridge or quasi-bridge, since otherwise the definition of $n^*$ would be violated: Indeed, $\o_{n^*-1}$ would have respectively a black bridge or black quasi-cross (having a black quasi-bridge in column $c^*$);
	\item a horizontal gray bridge, which cannot coexist with the black vertical quasi-bridge, in view of Lemma~\ref{lem:impossible}(ii);
	\item a horizontal gray quasi-bridge, since the black bridge in column $c^*$ could not be created with a single-site update, as illustrated in Figure~\ref{fig:2sne}.
\end{itemize}
\begin{figure}[!h]
	\centering
	\includegraphics[scale=0.95]{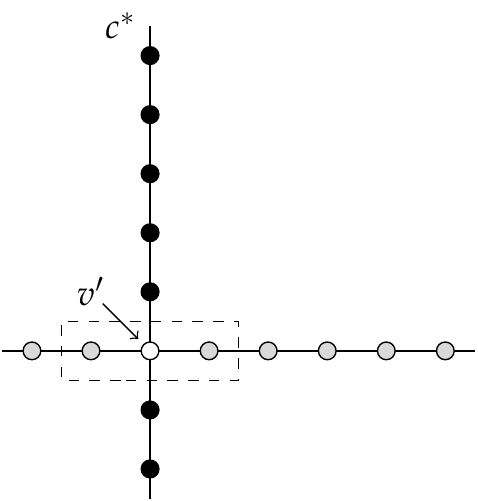}
	\caption{Three single-site updates are needed to create a black bridge in column $c^*$, since all the three sites in the dashed box must be updated}
	\label{fig:2sne}
\end{figure}
Therefore, by Lemma~\ref{lem:bqbi} $\D^{r_i}(\o_{n^*-1}) \geq 2$ for every $i=0,\dots,K-1$ and hence
\begin{equation}
\label{eq:U2K}
	\D(\o_{n^*-1}) = \sum_{i=0}^{K-1} \D^{r_i}(\o_{n^*-1}) \geq 2K.
\end{equation}
If $\D(\o_{n^*-1})\geq 2K+1$, then the proof is complete. Suppose instead that $\D(\o_{n^*-1})=2K$ and consider the configuration $\o_{n^*-2}$ preceding $\o_{n^*-1}$ in the path $\o$. By construction, $\o_{n^*-2}$ differs from $\o_{n^*-1}$ by a single-site update, and thus
\begin{equation}
\label{eq:pm1a}
	\D(\o_{n^*-2})=\D(\o_{n^*-1}) \pm1.
\end{equation}
Suppose first that $\D(\o_{n^*-2})=\D(\o_{n^*-1})-1 = 2K-1$, which means that configuration $\o_{n^*-2}$ has exactly one more particle than configuration $\o_{n^*-1}$. The site where such a particle is added cannot be $v'$, otherwise the definition of $n^*$ would be violated. All the other sites in column $c^*$ are already occupied, hence $\o_{n^*-2}$ is identical to $\o_{n^*-1}$ in column $c^*$. In particular, $\o_{n^*-2}$ has a black quasi-bridge in column $c^*$ as well. The configuration $\o_{n^*-2}$ cannot have any horizontal bridge, since the existence of a black bridge would contradict the definition of $n^*$ and that of a gray bridge is impossible by Lemma~\ref{lem:impossible}(ii). Since $\D(\o_{n^*-2})=2K-1$, by the pigeonhole principle there exists then at least one row, say $r'$, with $\D^{r'}(\o_{n^*-2})=1$, which means that $\o_{n^*-2}$ has a horizontal quasi-bridge in row $r'$. Such a horizontal quasi-bridge cannot be black, otherwise $\o_{n^*-2}$ would have a quasi-cross, violating the definition of $n^*$. By Lemma~\ref{lem:impossible}(iii) and the presence of a black quasi-bridge in column $c^*$, a gray quasi-bridge could exist only in the row containing site $v'$. However, in this case it would then be impossible to obtain a black vertical bridge in only two single-site updates, since the minimum number of steps required is three, as illustrated in Figure~\ref{fig:2sne}. Therefore, it is not possible that $\D(\o_{n^*-2})=2K-1$ and, combining~\eqref{eq:U2K} and~\eqref{eq:pm1a}, we deduce that
\[
	\D(\o_{n^*-2})=\D(\o_{n^*-1})+1 = 2K +1,
\]
which concludes the proof of case (a).\\

For case (b) we can argue as in case (a), but interchanging the role of rows and columns, and obtain that 
\[
	\max\{\D(\o_{n^*-1}), \D(\o_{n^*-2})\} \geq 2L +1.
\]

For case (c), let $r^*$ and $c^*$ be respectively the row and the column where the black quasi-cross lies in configuration $\o_{n^*}$ and let $v^*$ the site where they intersect. We distinguish two scenarios: (c1) the quasi-cross has exactly one empty site, which has to be $v^*$, and (c2) the quasi-cross has exactly two empty sites both different from $v^*$. Figure~\ref{fig:c1c2} illustrates these two possible scenarios.

\begin{figure}[!h]
	\centering
	\subfloat[$\o_{n^*}$ in case (c1)\label{fig:c1o}]{\includegraphics[scale=0.95]{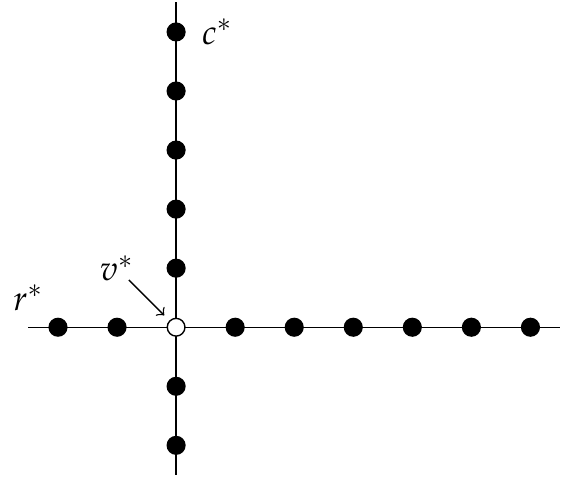}}
	\hspace{1.5cm}
	\subfloat[$\o_{n^*}$ in case (c2)\label{fig:c2o}]{\includegraphics[scale=0.95]{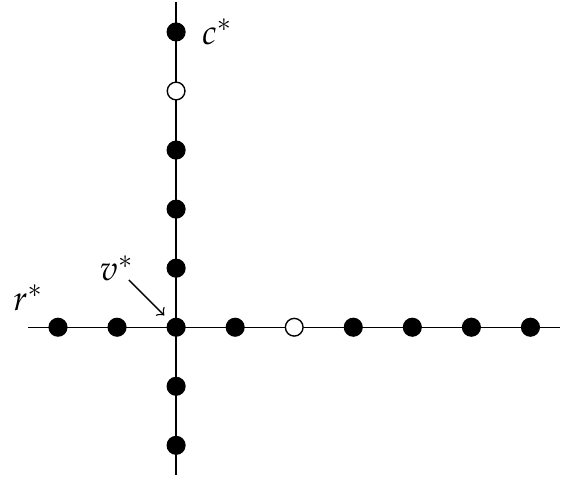}}
	\caption{Schematic representation of the two possible scenarios when the configuration $\o_{n^*}$ displays a black quasi-cross}
	\label{fig:c1c2}
\end{figure}
\FloatBarrier

Consider scenario (c1) first. The previous configuration $\o_{n^*-1}$ along the path $\o$ differs from $\o_{n^*}$ in exactly one site, say $v'$. By definition of $n^*$, configuration $\o_{n^*-1}$ does not display a quasi-cross, so such a site $v'$ has to lie either in row $r^*$, to which we refer as case (c1.i) (see Figure~\ref{fig:c1a}), or in column $c^*$, to which we refer as case (c1.ii) (see Figure~\ref{fig:c1b}). Furthermore, note that the quasi-cross present in $\o_{n^*}$ could only have been created by the addition of a black particle, hence $\o_{n^*-1}(v')=0$ and $\o_{n^*}(v')=\bb(v')$.

\begin{figure}[!h]
	\centering
	\subfloat[$\o_{n^*-1}$ in scenario (c1.i)\label{fig:c1a}]{\includegraphics[scale=0.95]{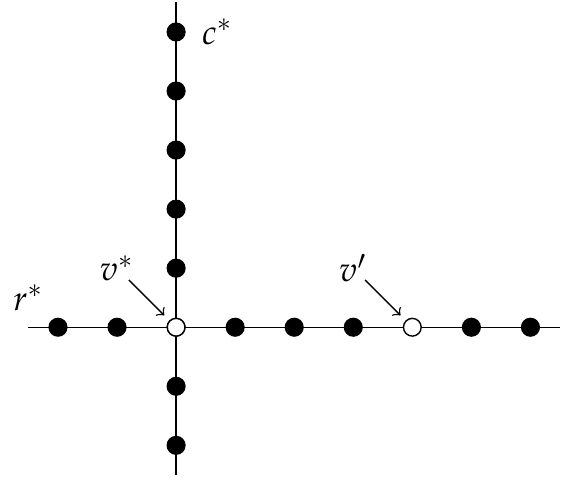}}
	\hspace{1.5cm}
	\subfloat[$\o_{n^*-1}$ in scenario (c1.ii)\label{fig:c1b}]{\includegraphics[scale=0.95]{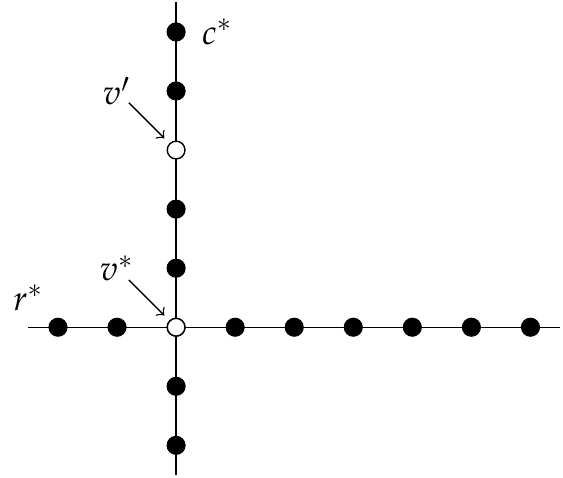}}
	\caption{Schematic representation of configuration $\o_{n^*-1}$ in case (c1)}
	\label{fig:c1}
\end{figure}
\FloatBarrier

In case (c1.i) $\o_{n^*-1}$ is such that $v'$ lies in row $r^*$, as in Figure~\ref{fig:c1a}. Then $\D^{r^*}(\o_{n^*-1})=2$, since row $r^*$ has exactly two empty sites, and $\D^{r}(\o_{n^*-1})\geq 2$ for all rows $r\neq r^*$, since none of them can display a black bridge or quasi-bridge (by definition of $n^*$) and neither a gray bridge or quasi bridge (by Lemma~\ref{lem:impossible}). Hence,
\[
	\D(\o_{n^*-1}) = \sum_{i=0}^{K-1} \D^{r_i}(\o_{n^*-1}) \geq 2 K.
\]
If $\D(\o_{n^*-1}) \geq 2K +1$, then the proof of case (c1.a) is complete. Suppose instead that $\D(\o_{n^*-1}) = 2K$ and consider the configuration $\o_{n^*-2}$ preceding $\o_{n^*-1}$ in the path $\o$. By construction, the configuration $\o_{n^*-2}$ differs from $\o_{n^*-1}$ by a single-site update and thus
\begin{equation}
\label{eq:c1apm1}
	\D(\o_{n^*-2})=\D(\o_{n^*-1})\pm1.
\end{equation}
Consider the case where $\D(\o_{n^*-2}) = \D(\o_{n^*-1})-1 = 2K - 1$, which means that configuration $\o_{n^*-2}$ has exactly one more particle than configuration $\o_{n^*-1}$. By virtue of the pigeonhole principle, the configuration $\o_{n^*-2}$ must have at least one row, say $r'$, such that $\D^{r'}(\o_{n^*-2})\leq 1$. In view of Lemma~\ref{lem:bqbi}, $\o_{n^*-2}$ then has to display a bridge or a quasi-bridge in row $r'$, which leads to a contradiction, since on this row $\o_{n^*-2}$ cannot have
\begin{itemize}
	\item a black horizontal bridge, by definition of $n^*$;
	\item a black horizontal quasi-bridge, since otherwise $\o_{n^*-2}$ would have a quasi-cross in row $r'$ and column $c^*$, violating the definition of $n^*$;
	\item a gray horizontal bridge or quasi-bridge, due to the presence of the black vertical quasi-bridge in column $c^*$ in view of Lemma~\ref{lem:impossible}.
\end{itemize} 
Hence, since $\D(\o_{n^*-1}) = 2K$, we deduce from~\eqref{eq:c1apm1} that
\[
	\D(\o_{n^*-2}) = 2K+1,
\]
which concludes the proof of case (c1.i).

In case (c1.ii) we can argue similarly to case (c1.i), by interchanging the role of rows and columns, and obtain that
\[
	\max\{\D(\o_{n^*-1}), \D(\o_{n^*-2})\} \geq 2L+1 \geq 2K+1.
\]

Consider now case (c2). We distinguish three scenarios, illustrated in Figure~\ref{fig:c2}, depending where the last particle (that created the quasi-cross in configuration $\o_{n^*}$) has been added: (c2.i) in a site $v'\neq v^*$ in row $r^*$ or (c2.ii) in a site $v'\neq v^*$ in column $c^*$ or (c2.iii) in the site $v'=v^*$.

\begin{figure}[!ht]
	\centering
	\subfloat[$\o_{n^*-1}$ in scenario (c2.i)\label{fig:c2a}]{\includegraphics[scale=0.945]{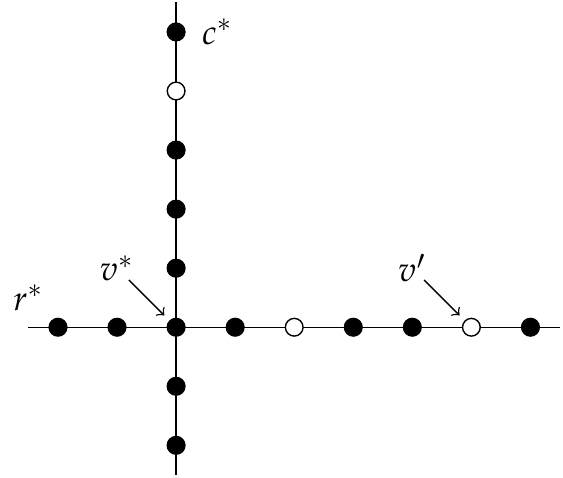}}
	\hspace{0.295cm}
	\subfloat[$\o_{n^*-1}$ in scenario (c2.ii)\label{fig:c2b}]{\includegraphics[scale=0.945]{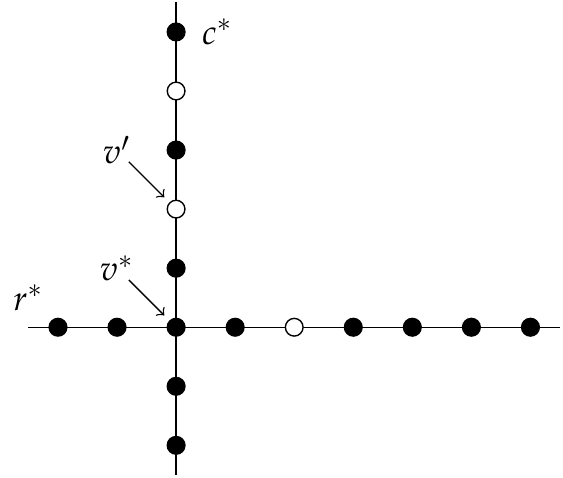}}
	\hspace{0.295cm}
	\subfloat[$\o_{n^*-1}$ in scenario (c2.iii)\label{fig:c2c}]{\includegraphics[scale=0.945]{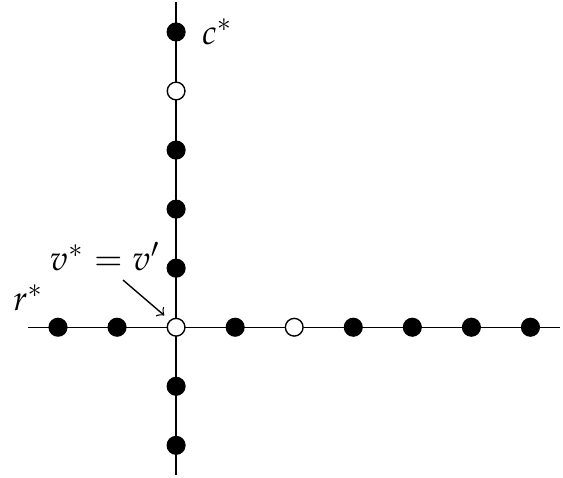}}
	\caption{Schematic representation of the three possible scenarios in case (c2)}
	\label{fig:c2}
\end{figure}
\FloatBarrier

In case (c2.i), let $c'$ be the column where site $v'$ lies. We first notice that configuration $\o_{n^*-1}$ cannot have a vertical gray bridge or quasi-bridge in column $c'$, since otherwise it would not be possible to add a black particle in site $v'$ in a single step, see Figure~\ref{fig:c2a}. Moreover $\o_{n^*-1}$ has no horizontal black quasi-bridges, since any of them would create, together with column $c^*$, a quasi-cross, violating the definition of $n^*$.

\begin{figure}[!h]
	\centering
	\includegraphics[scale=0.95]{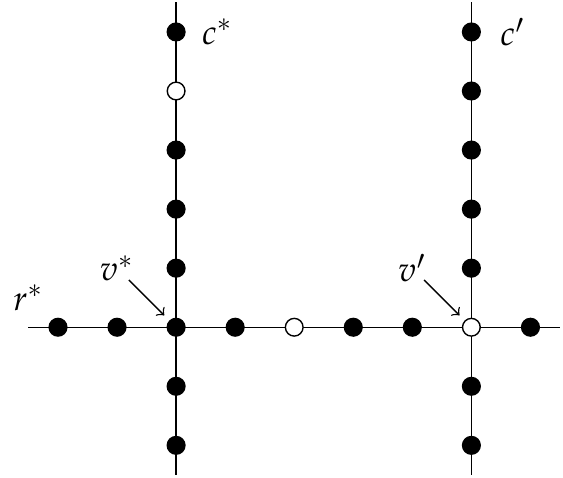}
	\caption{Schematic representation of configuration $\o_{n^*-1}$ in scenario (c2.i) with a quasi-bridge in column $c'$}
	\label{fig:c2afirstcase}
\end{figure}
\FloatBarrier
Suppose first that configuration $\o_{n^*-1}$ has a vertical black quasi-bridge in column $c'$, as in Figure~\ref{fig:c2afirstcase}, which means that
\begin{equation}
\label{eq:uc1}
	\D^{c'}(\o_{n^*-1}) =1.
\end{equation}
By virtue of Lemma~\ref{lem:impossible}, there cannot be any horizontal gray bridges or quasi-bridges. Furthermore, $\o_{n^*-1}$ has no horizontal black quasi-bridges, which would create a quasi-cross together with column $c'$, violating again the definition of $n^*$. In view of Lemma~\ref{lem:bqbi}, $\D^r(\o_{n^*-1}) \geq 2$ for every row $r$ and thus
\[
	\D(\o_{n^*-1}) \geq 2K.
\]
If $\D(\o_{n^*-1}) \geq 2K +1$, then the proof is complete. Suppose instead that $\D(\o_{n^*-1}) = 2K$ and consider the configuration $\o_{n^*-2}$ preceding $\o_{n^*-1}$ in the path $\o$. By construction, $\o_{n^*-2}$ differs from $\o_{n^*-1}$ by a single-site update and thus
\begin{equation}
\label{eq:c1apm1_bis}
	\D(\o_{n^*-2})=\D(\o_{n^*-1})\pm1.
\end{equation}
Consider the case where $\D(\o_{n^*-2}) = \D(\o_{n^*-1})-1 = 2K - 1$, which means that configuration $\o_{n^*-2}$ has exactly one more particle than configuration $\o_{n^*-1}$. By virtue of the the pigeonhole principle, the configuration $\o_{n^*-2}$ must have at least one row, say $r'$, such that $\D^{r'}(\o_{n^*-2})\leq 1$. In view of Lemma~\ref{lem:bqbi}, $\o_{n^*-2}$ has then to display a bridge or a quasi-bridge in row $r'$. If $r'=r^*$, then $\o_{n^*-2}$ would have a black quasi-cross or a black bridge, violating the definition of $n^*$. If $r' \neq r^*$, then we also obtain a contradiction, since in row $r'$ $\o_{n^*-2}$ cannot have
\begin{itemize}
	\item a black bridge, by definition of $n^*$;
	\item a black quasi-bridge, since otherwise $\o_{n^*-2}$ would have a quasi-cross in row $r'$ and column $c^*$, violating the definition of $n^*$;
	\item a gray bridge or quasi-bridge, due to the presence of a black quasi-bridge in column $c^*$ and Lemma~\ref{lem:impossible}.
\end{itemize} 
Hence, it cannot be the case that $\D(\o_{n^*-2}) = 2K - 1$, and from~\eqref{eq:c1apm1_bis} it follows that $\D(\o_{n^*-2}) = 2K+1.$

Suppose now that configuration $\o_{n^*-1}$ does not have a vertical black quasi-bridge in column $c'$. By virtue of Lemma~\ref{lem:bqbi}, $\D^{c'}(\o_{n^*-1}) \geq 2$. We first consider the case where
\begin{equation}
\label{eq:uc2}
	\D^{c'}(\o_{n^*-1}) =2,
\end{equation}
so that $\D^{c'}(\o_{n^*}) = 1$, which means that $\o_{n^*}$ has an additional quasi-cross, namely the one lying in row $r^*$ and column $c'$, see Figure~\ref{fig:c2atwocrosses}. In this case, we can conclude the proof by looking at this other quasi-cross and arguing as in sub-case (3) of scenario (c2.iii), which will be presented later.

\begin{figure}[!h]
	\centering
	\subfloat{\includegraphics[scale=0.95]{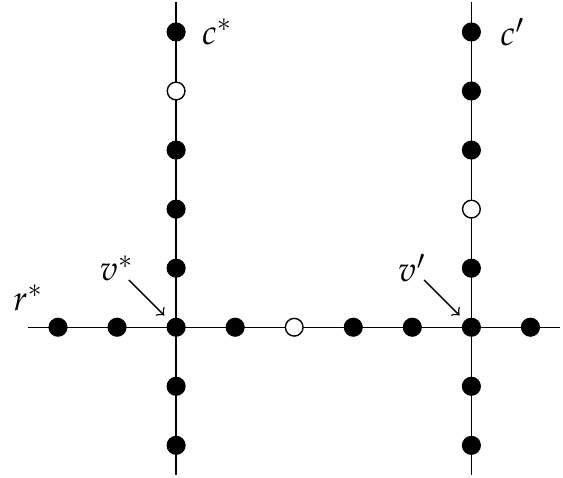}}
	\hspace{1.5cm}
	\subfloat{\includegraphics[scale=0.95]{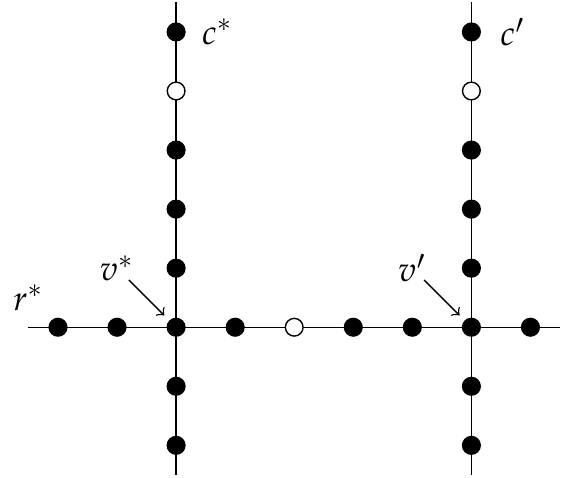}}
	\caption{Schematic representation of configuration $\o_{n^*}$ with two quasi-crosses}
	\label{fig:c2atwocrosses}
\end{figure}
\FloatBarrier

Therefore, in view of~\eqref{eq:uc1} and~\eqref{eq:uc2}, we can assume
\begin{equation}
\label{eq:Ucgeq3}
	\D^{c'}(\o_{n^*-1})\geq 3.
\end{equation}
We then distinguish three sub-cases, depending on whether $\o_{n^*-1}$ has (1) no vertical quasi-bridges (see Figure~\ref{fig:c2i1}) or (2) at least one gray vertical quasi-bridge and no black vertical quasi-bridges (see Figure~\ref{fig:c2i2}) or (3) at least one black vertical quasi-bridge (see Figure~\ref{fig:c2i3}).

\begin{figure}[!h]
	\centering
	\subfloat[$\o_{n^*-1}$ in sub-case (1)\label{fig:c2i1}]{\includegraphics[scale=0.95]{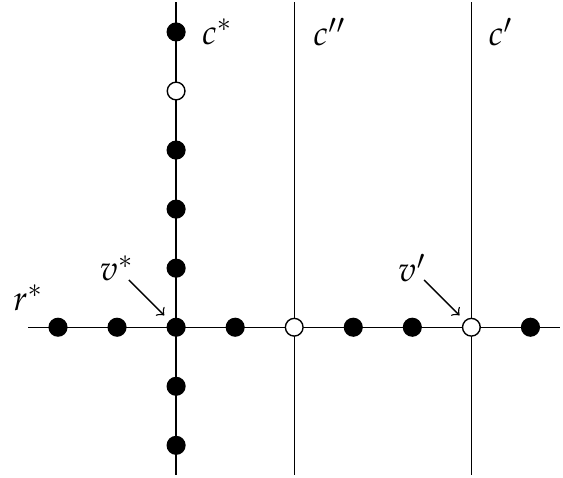}}
	\hspace{1.5cm}
	\subfloat[$\o_{n^*-1}$ in sub-case (2)\label{fig:c2i2}]{\includegraphics[scale=0.95]{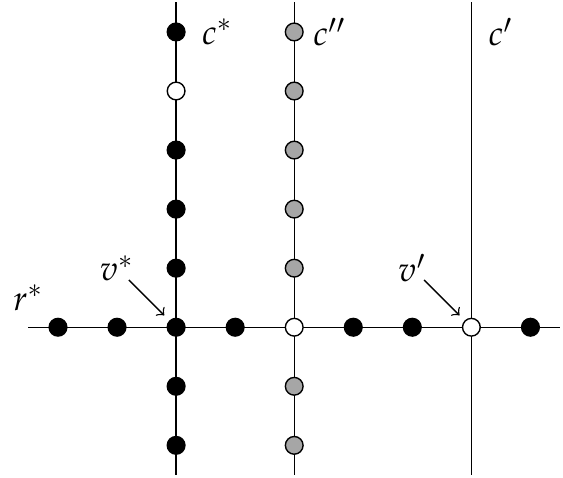}}
	\caption{Schematic representation of the sub-cases (1) and (2) of scenario (c2.i) when condition~\eqref{eq:Ucgeq3} is satisfied}
\end{figure}
\FloatBarrier

In sub-case (1), $\o_{n^*-1}$ has no vertical quasi-bridges except the one in column $c^*$, so by Lemma~\ref{lem:bqbi} $\D^{c}(o_{n^*-1}) \geq 2$ for every $c \neq c^*$. This fact and~\eqref{eq:Ucgeq3} yield
\[
	\D(\o_{n^*-1}) = \sum_{j=0}^{L-1} \D^{c_j}(\o_{n^*-1}) \geq 2L,
\]
and the proof of sub-case (1) is completed.\\

In sub-case (2), $\o_{n^*-1}$ can have only one vertical gray quasi-bridge and it must lie in column $c''$ by Lemma~\ref{lem:impossible}, see Figure~\ref{fig:c2i2}. Lemma~\ref{lem:bqbi} gives $\D^{c''}(\o_{n^*-1}) =1$. All the columns $c\neq c^*,c'', c'''$ do not display a vertical quasi-bridge, so $\D^{c}(\o_{n^*-1}) \geq 2$ by Lemma~\ref{lem:bqbi}. These facts and~\eqref{eq:Ucgeq3} yield
\[
	\D(\o_{n^*-1}) = \sum_{j=0}^{L-1} \D^{c_j}(\o_{n^*-1}) \geq 2L-1.
\]
If $\D(\o_{n^*-1}) \geq 2L$, the proof is complete. Suppose instead that $\D(\o_{n^*-1}) = 2L-1$ and consider the configuration $\o_{n^*-2}$ preceding $\o_{n^*-1}$ in the path $\o$. By construction, $\o_{n^*-2}$ differs from $\o_{n^*-1}$ by a single-site update and thus
\begin{equation}
\label{eq:c1apm1_bis3}
	\D(\o_{n^*-2})=\D(\o_{n^*-1})\pm1.
\end{equation}
Consider the case where $\D(\o_{n^*-2}) = \D(\o_{n^*-1})-1 = 2L - 2$, which means that configuration $\o_{n^*-2}$ has exactly one more particle than configuration $\o_{n^*-1}$. Such a particle cannot lie in site $v'$, otherwise $\o_{n^*-2}$ would have a quasi-cross, violating the definition of $n^*$. Furthermore, by virtue of the the pigeonhole principle, the configuration $\o_{n^*-2}$ must have at least two horizontal quasi-bridges or one horizontal bridge, which cannot exists neither black or gray by Lemma~\ref{lem:impossible} due to the presence of the black quasi-bridge in column $c^*$ and the gray quasi-bridge in column $c''$. Hence, $\D(\o_{n^*-2}) \neq 2L - 2$, and from~\eqref{eq:c1apm1_bis3} it follows that 
\[
	\D(\o_{n^*-2}) = 2L,
\]
which completes the proof of sub-case (2).

Consider now sub-case (3), which is illustrated in Figure~\ref{fig:c2i3}.
\captionsetup[subfigure]{labelformat=empty}
\begin{figure}[!ht]
	\centering
	\subfloat[Vertical black quasi-bridge in a column different from $c''$ with empty site aligned with that of $c^*$]{\includegraphics[scale=0.945]{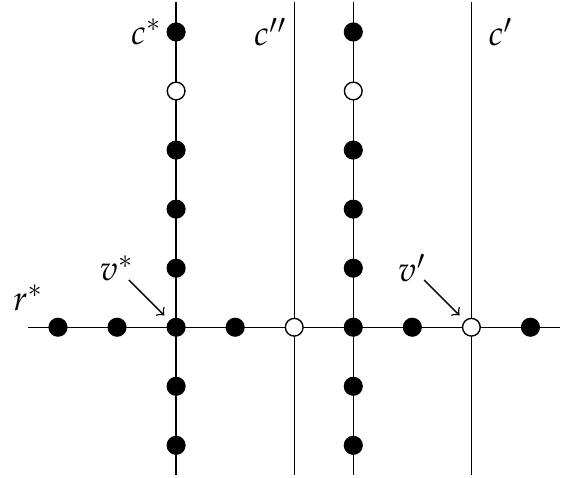}}
	\hspace{0.295cm}
	\subfloat[Vertical black quasi-bridge in a column different from $c''$ with empty site not aligned with that of $c^*$]{\includegraphics[scale=0.945]{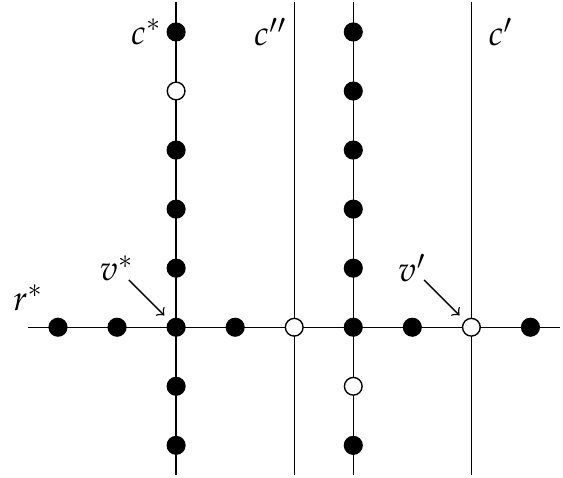}}
	\hspace{0.295cm}
	\subfloat[Vertical black quasi-bridge in column $c''$\label{fig:c2ai}]{\includegraphics[scale=0.945]{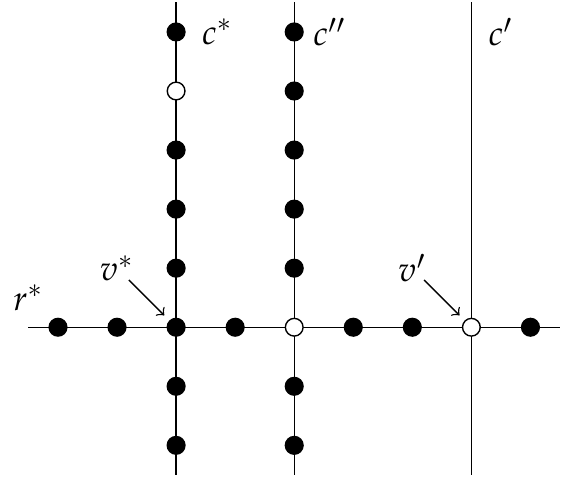}}	
	\caption{Schematic representation of $\o_{n^*-1}$ in sub-case (3) of scenario (c2.i) when condition~\eqref{eq:Ucgeq3} is satisfied}
	\label{fig:c2i3}
\end{figure}
\captionsetup[subfigure]{labelformat=parens}
\FloatBarrier

Since none of the rows can display a bridge or a quasi-bridge without contradicting Lemma~\ref{lem:impossible} or violating the definition of $n^*$, it follows from Lemma~\ref{lem:bqbi} that
\[
	\D(\o_{n^*-1}) \geq 2K.
\]
If $\D(\o_{n^*-1}) \geq 2K +1$, the proof of the sub-case is complete. Consider now the remaining case, namely $\D(\o_{n^*-1}) = 2K$. Consider the configuration $\o_{n^*-2}$ preceding $\o_{n^*-1}$ in the path $\o$. By construction, the configuration $\o_{n^*-2}$ differs from $\o_{n^*-1}$ by a single-site update and thus
\begin{equation}
\label{eq:c2aiiipm1}
	\D(\o_{n^*-2})=\D(\o_{n^*-1})\pm1.
\end{equation}
Consider the case where $\D(\o_{n^*-2}) = \D(\o_{n^*-1})-1 = 2K - 1$, so that configuration $\o_{n^*-2}$ has exactly one more particle than configuration $\o_{n^*-1}$ (in a site which cannot be $v'$). Thanks to the pigeonhole principle, the configuration $\o_{n^*-2}$ must have at least one row with energy difference strictly less than $2$. On such a row $\o_{n^*-2}$ then has to display a bridge or a quasi-bridge, thanks to Lemma~\ref{lem:bqbi}.

Since $\o_{n^*-2}$ has all the vertical black quasi-bridges that $\o_{n^*-1}$ has, it is impossible for $\o_{n^*-2}$ to display
\begin{itemize}
	\item a black horizontal bridge, by definition of $n^*$;
	\item a black horizontal quasi-bridge, since otherwise it would create a quasi-cross together with the vertical quasi-bridge in column $c^*$, violating the definition of $n^*$;
	\item a gray horizontal bridge by Lemma~\ref{lem:impossible}, due to the presence of the vertical black quasi-bridge in column $c^*$, see Figure~\ref{fig:c2i3};
	\item a gray horizontal quasi-bridge, since every row has either a black particle or two empty sites, see Figure~\ref{fig:c2i3}.
\end{itemize} 
Hence, $\D(\o_{n^*-2}) \neq 2K - 1$, and from~\eqref{eq:c2aiiipm1} it follows that 
\[
	\D(\o_{n^*-2}) = 2K+1,
\]
which completes the proof of sub-case (3).

In scenario (c2.ii), the configuration $\o_{n^*-1}$ preceding $\o_{n^*}$ along the path $\o$ cannot have any vertical quasi-bridge, since it would create a quasi-cross together with row $r^*$. We distinguish two sub-cases, depending on whether $\o_{n^*-1}$ has (1) no gray vertical quasi-bridges, see Figure~\ref{fig:c2ii1}, or (2) at least one gray vertical quasi-bridge, see Figure~\ref{fig:c2ii2}.
\begin{figure}[!h]
	\centering
	\subfloat[$\o_{n^*-1}$ in sub-case (1)\label{fig:c2ii1}]{\includegraphics[scale=0.95]{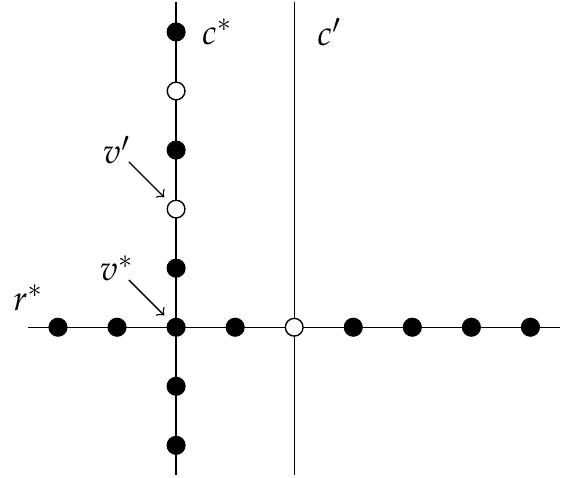}}
	\hspace{1.5cm}
	\subfloat[$\o_{n^*-1}$ in sub-case (2)\label{fig:c2ii2}]{\includegraphics[scale=0.95]{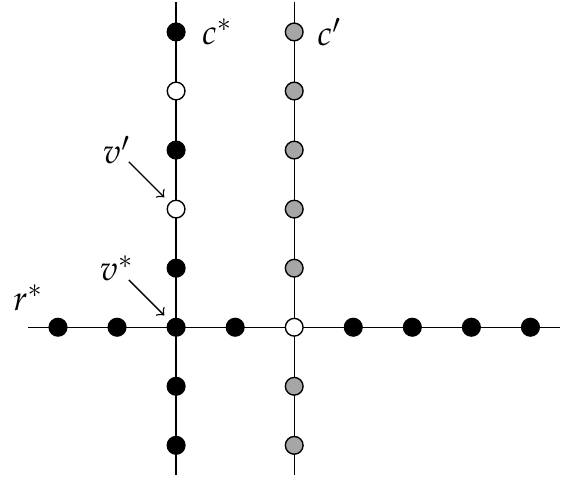}}
	\caption{Schematic representation of the two sub-cases in scenario (c2.ii)}
\end{figure}
\FloatBarrier

In sub-case (1), by Lemma~\ref{lem:bqbi} $\D^{c}(\o_{n^*-1}) \geq 2$ in every column $c$ and hence 
\[
	\D(\o_{n^*-1}) \geq 2L.
\]
In sub-case (2), it is clear that there can be only one gray vertical quasi-bridge, which intersects row $r^*$ in the empty site different from $v^*$. By looking at the energy difference in columns and arguing similarly to the final part of sub-case (2) of scenario (c2.i), we can conclude that 
\[
	\max\{ \D(\o_{n^*-1}), \D(\o_{n^*-2})\} \geq 2L.
\]

Consider now scenario (c2.iii). By definition of $n^*$, $\o_{n^*-1}$ does not have any black bridge and by Lemma~\ref{lem:bqbi} it cannot have any gray bridge either. We distinguish three sub-cases, depending on whether $\o_{n^*-1}$ has (1) no vertical quasi-bridges, see Figure~\ref{fig:c2iii1}, (2) gray vertical quasi-bridges, but no black vertical quasi-bridges, see Figure~\ref{fig:c2iii2}, or (3) black vertical quasi-bridges, see Figure~\ref{fig:c2iii3}.

In sub-case (1), by Lemma~\ref{lem:bqbi} $\D^{c}(\o_{n^*-1}) \geq 2$ in every column $c$ and hence 
\[
	\D(\o_{n^*-1}) \geq 2L.
\]
In sub-case (2), $\o_{n^*-1}$ can have only one vertical gray quasi-bridge and it must lie in column $c'$ by Lemma~\ref{lem:impossible}, see Figure~\ref{fig:c2iii2}. Lemma~\ref{lem:bqbi} gives $\D^{c'}(o_{n^*-1}) =1$. All the columns $c\neq c'$ do not display a vertical quasi-bridge, so $\D^{c}(o_{n^*-1}) \geq 2$ by Lemma~\ref{lem:bqbi}. These facts and~\eqref{eq:Ucgeq3} yield
\[
	\D(\o_{n^*-1}) = \sum_{j=0}^{L-1} \D^{c_j}(o_{n^*-1}) \geq 2L-1.
\]
If $K<L$, then $\D(\o_{n^*-1}) \geq 2L-1 \geq 2K+1$ and the proof of sub-case (2) is complete. If $K=L$ and $\D(\o_{n^*-1}) \geq 2L$, then the proof is also complete. Suppose instead that $K=L$ and $\D(\o_{n^*-1}) = 2L-1=2K-1$ and consider the configuration $\o_{n^*-2}$ preceding $\o_{n^*-1}$ in the path $\o$. By construction, $\o_{n^*-2}$ differs from $\o_{n^*-1}$ by a single-site update and thus
\begin{equation}
\label{eq:c1apm1_bisc2iii2}
	\D(\o_{n^*-2})=\D(\o_{n^*-1})\pm1.
\end{equation}
Consider the case where $\D(\o_{n^*-2}) = \D(\o_{n^*-1})-1 = 2K - 2$, which means that configuration $\o_{n^*-2}$ has exactly one more particle than configuration $\o_{n^*-1}$. By virtue of the the pigeonhole principle, the configuration $\o_{n^*-2}$ must have at least two horizontal quasi-bridges or one horizontal bridge, which cannot exists neither black or gray by Lemma~\ref{lem:impossible} due to the presence of the black quasi-bridge in column $c^*$ and the gray quasi-bridge in column $c'$. Hence, $\D(\o_{n^*-2}) \neq 2K - 1$, and from the fact that $\D(\o_{n^*-1}) = 2K-1$ and~\eqref{eq:c1apm1_bisc2iii2} it follows that 
\[
	\D(\o_{n^*-2}) = 2K,
\]
which completes the proof of sub-case (2).

\begin{figure}[!h]
	\centering
	\subfloat[$\o_{n^*-1}$ in sub-case (1)\label{fig:c2iii1}]{\includegraphics[scale=0.95]{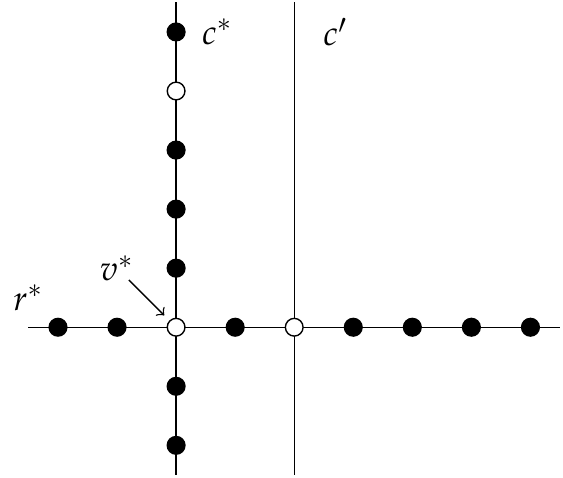}}
	\hspace{1.5cm}
	\subfloat[$\o_{n^*-1}$ in sub-case (2)\label{fig:c2iii2}]{\includegraphics[scale=0.95]{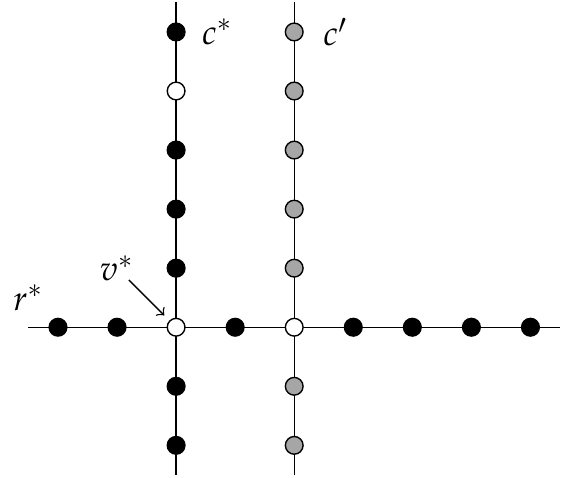}}
	\caption{Schematic representation of sub-cases (1) and (2) in scenario (c2.iii)}
\end{figure}
\FloatBarrier

In sub-case (3), illustrated in Figure~\ref{fig:c2iii3} the presence of a vertical black quasi-bridge in a column $c'' \neq c^*$ means that there are no horizontal quasi-bridges in $\o_{n^*-1}$. Indeed the presence of a horizontal black quasi-bridge together with the quasi-bridge in column $c''$ would create a quasi-cross, violating the definition of $n^*$. Furthermore, it is impossible for $\o_{n^*-1}$ to have a horizontal gray quasi-bridge, since every row has either a black particle or two empty sites, see Figure~\ref{fig:c2iii3}. 

\captionsetup[subfigure]{labelformat=empty}
\begin{figure}[!h]
	\centering
	\subfloat[$\o_{n^*-1}$ with the empty site of column $c'$ in row $r^*$]{\includegraphics[scale=0.95]{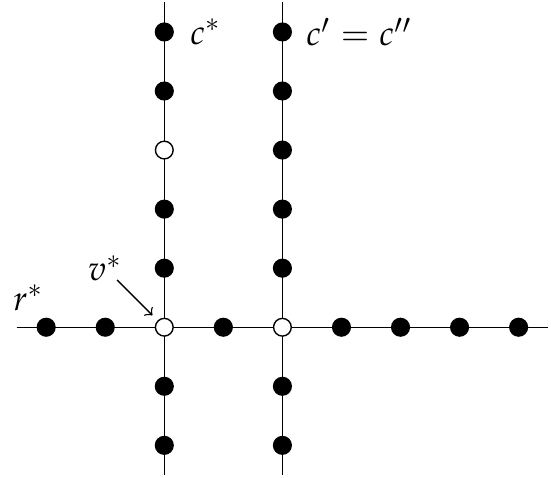}}
	\hspace{0.4cm}
	\subfloat[$\o_{n^*-1}$ with the empty site of column $c'$ aligned with that of column $c^*$]{\includegraphics[scale=0.95]{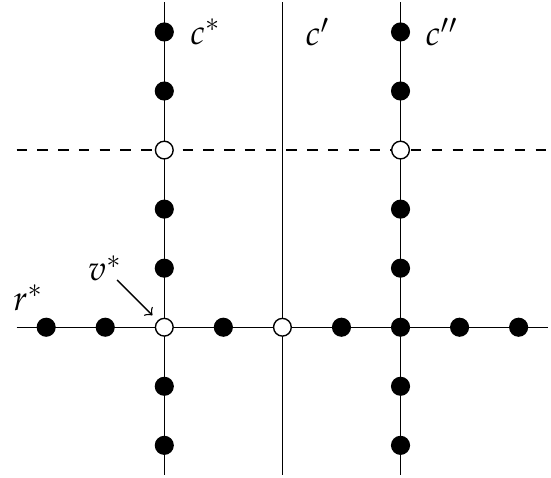}}
	\hspace{0.4cm}
	\subfloat[$\o_{n^*-1}$ with the empty site of column $c'$ not in row $r^*$ and not aligned with that of column $c^*$]{\includegraphics[scale=0.95]{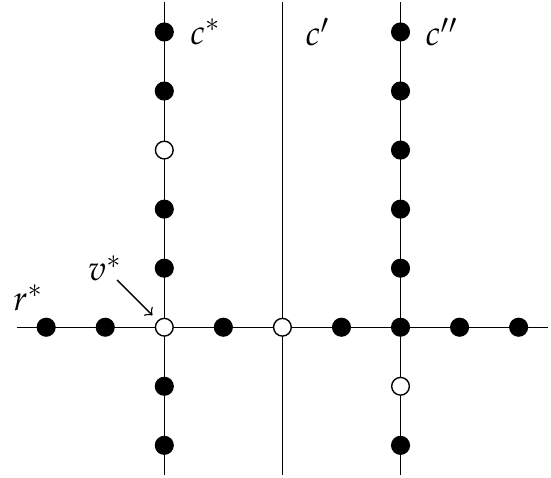}}
	\caption{Schematic representation of configuration $\o_{n^*-1}$ in sub-case (3) of (c2.iii)}
	\label{fig:c2iii3}
\end{figure}
\captionsetup[subfigure]{labelformat=parens}
\FloatBarrier

Since trivially $\o_{n^*-1}$ displays no bridges either, by applying Lemma~\ref{lem:bqbi} to the rows, we get
\[
	\D(\o_{n^*-1}) \geq 2K.
\]
If $\D(\o_{n^*-1}) \geq 2K +1$, then the proof of sub-case (3) is complete. Suppose instead that $\D(\o_{n^*-1}) = 2K$ and consider the configuration $\o_{n^*-2}$ preceding $\o_{n^*-1}$ in the path $\o$. By construction, the configuration $\o_{n^*-2}$ differs from $\o_{n^*-1}$ by a single-site update and thus
\begin{equation}
\label{eq:c2iii3pm1}
	\D(\o_{n^*-2})=\D(\o_{n^*-1})\pm1.
\end{equation}
Consider the case where $\D(\o_{n^*-2}) = \D(\o_{n^*-1})-1 = 2K - 1$, so that configuration $\o_{n^*-2}$ has exactly one more particle than configuration $\o_{n^*-1}$. Thanks to the pigeonhole principle, the configuration $\o_{n^*-2}$ must have at least one row, say $r'$ with energy difference strictly less than $2$. On such a row $\o_{n^*-2}$ then has to display a bridge or a quasi-bridge, thanks to Lemma~\ref{lem:bqbi}. This leads to a contradiction, since $\o_{n^*-2}$ cannot have in row $r'$ 
\begin{itemize}
	\item a black bridge, by definition of $n^*$;
	\item a black quasi-bridge, since otherwise $\o_{n^*-2}$ would have a quasi-cross in row $r'$ and column $c''$, violating the definition of $n^*$;
	\item a gray bridge or quasi-bridge, since every row has either a black particle or two empty sites, see Figure~\ref{fig:c2iii3}.
\end{itemize} 
Hence, by~\eqref{eq:c2iii3pm1} we obtain that 
\[
	\D(\o_{n^*-2}) = 2K+1. \qedhere
\]
\end{proof}

\subsection*{Reduction algorithm for WR configurations (periodic boundary conditions)}
We describe now an iterative procedure, to which we will refer as \textit{reduction algorithm} that builds a path $\o$ from a suitable initial WR configuration $\s$ to any of the $M$ stable configurations. Assume that the target stable configuration is $\bbb \in \ss$ and that $\bbb$ is the configurations with all sites occupied by particles of type $b$, to which we refer as \textit{black particles}. 

The algorithm outputs a path in which black particles are progressively added column by column to the original configuration $\s$ removing all non-black particles when necessary. In order to be able to start this procedure, we require that the initial configuration $\s$ is such that all the sites in the first vertical stripe $C_1$ are either empty or occupied by black particles, \ie
\begin{equation}
\label{eq:racond_toricgrid}
	\s(v) \in \{0,b\} \quad \forall \, v \in C_1.
\end{equation}
The desired path $\o$ is constructed in such a way that the maximum energy achieved along the path is $H(\s)+1$ and is built as the concatenation of $L$ paths $\o^{(1)}, \dots, \o^{(L)}$. The intuition is that for every $j=1,\dots,L$ along path $\o^{(j)}$ the non-black particles are removed from column $c_{j+1}$ and simultaneously black particles are progressively added on column $c_{j}$. More specifically, the path $\o^{(j)}$ goes from $\s_{j}$ to $\s_{j+1}$, where we define $\s_1=\s$ and for $j=2,\dots, L+1$
\begin{equation}
\label{eq:oj}
	\s_{j}(v) :=
	\begin{cases}
		b & \text{ if } v \in \bigcup_{i=1}^{j-1} c_i,\\
		0 	& \text{ if } v \in c_{j} \text{ and } \s(v)\neq b,\\
		\s(v) & \text{ if } v \in c_{j} \text{ and } \s(v)=b \text{ or } v \in \bigcup_{i=j+1}^{L} c_i.
	\end{cases}
\end{equation}
Clearly, due to the periodic boundary conditions, the column indices $L$ and $0$ should be identified. It is easy to check that $\s_{L+1}=\bbb$. We now describe in detail how to construct each of the paths $\o^{(j)}$ for $j=1,\dots,L$. We build a path $\o^{(j)}=(\o^{(j)}_1, \dots, \o^{(j)}_{2K+1})$ of length $2K+1$ (but possibly with void moves), with $\o^{(j)}_1=\s_j$ and $\o^{(j)}_{2K+1}=\s_{j+1}$. We repeat iteratively the following procedure for every $i=1,\dots,2K$:
\begin{itemize}
\item If $i \equiv 1 \pmod 2$, consider site $v=(j+1, (i-1)/2)$:
	\begin{itemize}
		\item[-] If $\o^{(j)}_i(v) \in \{0,b\}$, then set $\o^{(j)}_{i+1}(v)=\o^{(j)}_{i}(v)$;
		\item[-] If $\o^{(j)}_i(v)\not\in \{0, b\}$, then remove the non-black particle in site $v$ from configuration $\o^{(j)}_i$, obtaining in this way a new configuration $\o^{(j)}_{i+1}$ with $\o^{(j)}_{i+1}(v)=0$ and thus $H(\o^{(j)}_{i+1}) = H(\o^{(j)}_{i})+1$.
	\end{itemize}
\item If $i \equiv 0 \pmod 2$, consider site $v=(j,i/2-1)$:
	\begin{itemize}
		\item[-] If $\o^{(j)}_i(v)=b$, then set $\o^{(j)}_{i+1}(v)=\o^{(j)}_i(v)=b$;
		\item[-] If $\o^{(j)}_i(v)=0$, then add a black particle in site $v$, obtaining in this way a new configuration $\o^{(j)}_{i+1}$ such that $\o^{(j)}_{i+1}(v)=b$ and $H(\o^{(j)}_{i+1}) = H(\o^{(j)}_{i})-1$. This is WR configuration because, by construction, there are no particles of different type in any neighboring sites of $v$. In particular, the site at the right of $v$ has possibly been emptied at the previous step, if it was occupied by a non-black particle.
	\end{itemize}
\end{itemize}
Note that there are no non-black particles in $c_0$ in configuration $\s$ by virtue of~\eqref{eq:racond_toricgrid} and this properties is inherited by all the configurations $\s_1,\dots,\s_{L}$. As a consequence, in the last path $\o^{(L)}$ all steps corresponding to odd values of $i$ are void. For every $j=1, \dots, L$, the configurations $\s_j$ and $\s_{j+1}$ satisfy the inequality
\[
	H(\s_{j+1}) \leq H(\s_j),
\]
since along the path $\o^{(j)}$ connecting them the number of black particles added in column $c_j$ is greater than or equal to the number of non-black particles removed in column $c_{j+1}$. Moreover,
\[
	\Phi_{\o^{(j)}} \leq H(\s_j) + 1,
\]
since along the path $\o^{(j)}$ every particle removal (if any) is always followed by a particle addition. These two properties imply that the path $\o: \s \to \bbb$ created by concatenating $\o^{(1)}, \dots, \o^{(L)}$ satisfies $\Phi_{\o} \leq H(\s)+1$, and thus 
\[
	\Phi(\s,\bbb)-H(\s) \leq 1.
\]
Note that this procedure can be adapted to construct a path from $\s \in \cX$ with target any other stable configuration, say $\bs^{(m)}$, $m=1,\dots,\mm$, provided that the following condition holds:
\begin{equation}
\label{eq:racond_toricgrid_bis}
	\s(v) \in \{0,m\} \quad \forall \, v \in C_1. 
\end{equation}
that is $\s$ has only empty sites or particles of type $m$ in the first vertical stripe.\\

We now show how the reduction algorithm can be used to build a reference path with prescribed energy height between any pair of stable configurations. The maximum energy barrier along such a reference path matches the right-hand side of inequality~\eqref{eq:lowertg_WR}, proving Theorem~\ref{thm:structuralpropertiesWR}(i). 

\begin{prop}[Reference path between stable WR configurations]\label{prop:refpathwrt}
Consider the $K \times L$ square lattice $\L$ with periodic boundary conditions with $(K,L) \neq (2,2), (2,3), (3,2)$. For every pair of stable configurations $\bs, \bsp \in \ss$ there exists a path $\o^*: \aa \to \bb$ in $\cX$ such that
\[
	\Phi_{\o^*} - H(\aa) = 
	\begin{cases}
		2K & \text{if } K=L,\\
		\min\{2K,2L\}+1 & \text{if } K\neq L.
	\end{cases}
\]
\end{prop}
\begin{proof}
In this proof we associate the color \textit{gray} to the type of particles present in configuration $\aa$ and the color \textit{black} to that present in configuration $\bb$. Without loss of generality we may assume that $K \leq L$. We then distinguish two cases, depending on whether (a) $K=L$ and (b) $K <L$.

Consider case (a) first. We will show that there exists a path $\o^*: \aa \to \bb$ such that $\Phi_{\o^*} - H(\aa) = 2K = 2L$. 
If $K=2$, $\o^*$ is simply the path that gradually removes all the four gray particles that $\aa$ has and then add four black particles one by one; one can immediately check that $\Phi_{\o^*} - H(\aa) = 4 = 2K$. We henceforth assume that $K\geq 3$. We construct the path $\o^*$ as concatenation of three paths, which we denote by $\o^{(1)}$, $\o^{(2)}$ and $\o^{(3)}$, respectively.

We first build a path $\o^{(1)}$ of length $4(K-1)+3$ as follows. First remove in two steps the gray particles in sites $(0,0)$ and $(0,K-1)$, obtaining configuration $\o^{(1)}_3$ with energy $H(\o^{(1)}) = H(\aa) +2$, as illustrated in Figure~\ref{fig:refpathWRtoric_omega1}. Then, iteratively define configuration $\o^{(1)}_{4+i}$ from $\o^{(1)}_{4+i-1}$ for $i=0,\dots,4(K-2)+3$ as follows:
\begin{itemize}
	\item If $i \equiv 0 \pmod 4$, then obtain configuration $\o^{(1)}_{4+i-1}$ from $\o^{(1)}_{4+i}$ by removing the gray particle in site $(L-1, \lfloor \frac{i}{4} \rfloor)$.
	\item If $i \equiv 1 \pmod 4$, then obtain configuration $\o^{(1)}_{4+i-1}$ from $\o^{(1)}_{4+i}$ by removing the gray particle in site $(1, \lfloor \frac{i}{4} \rfloor)$.
	\item If $i \equiv 2 \pmod 4$, then obtain configuration $\o^{(1)}_{4+i-1}$ from $\o^{(1)}_{4+i}$ by removing the gray particle in site $(0, \lfloor \frac{i}{4} \rfloor+1)$.
	\item If $i \equiv 3 \pmod 4$, then obtain configuration $\o^{(1)}_{4+i-1}$ from $\o^{(1)}_{4+i}$ by adding a black particle in site $(0, \lfloor \frac{i}{4} \rfloor)$.
\end{itemize}
Note that the step corresponding to $i=4(K-2)+2$ is void, since there is no particle in site $(0,K-1)$ to be removed, having been removed at the second step of $\o^{(1)}$. Denote by $\s':=\o^{(1)}_{4(K-1)+3}$ the configuration obtained by this procedure, which has energy difference 
\[
	\D(\s') = H(\s') - H(\aa) = 2K-1.
\]
The way the path $\o^{(1)}$ is built guarantees that
\begin{equation}
\label{eq:phiomega1_KequalL}
	\Phi_{\o^{(1)}} - H(\aa)= \max_{\h \in \o^{(1)}} H(\h) - H(\aa)= 2K.
\end{equation}

If $K = 3$, the second path $\o^{(2)}$ is not needed: Consider the configuration $\s'''$ obtained from $\s'$ by removing the gray particle in the site $(1,K-1)$ and thus $H(\s''')=H(\s')+1$. The configuration $\s'''$ satisfies the initial condition~\eqref{eq:racond_toricgrid}, so we can use the reduction algorithm to build the path $\o^{(3)}$ from $\s'''$ to $\bb$. The concatenation of $\o^{(1)}$ and $\o^{(3)}$ yields a path $\o^*$ from $\aa$ to $\bb$ such that \[\Phi_{\o^*}-H(\aa)=6=2K.\]

Therefore, we can assume that $K \geq 4$. In this case, we build a path $\o^{(2)}$ of length $K(K-4)+4$, which gradually enlarges column by column the quasi-bridge that configuration $\s'$ has until a ``black diamond'' is created, obtaining the configuration that we denote by $\s''$, see Figure~\ref{fig:refpathWRtoric_omega1}. 
\captionsetup[subfigure]{labelformat=empty}
\begin{figure}[!t]
	\centering
	\subfloat[$\o^{(1)}_3$]{\includegraphics[scale=1]{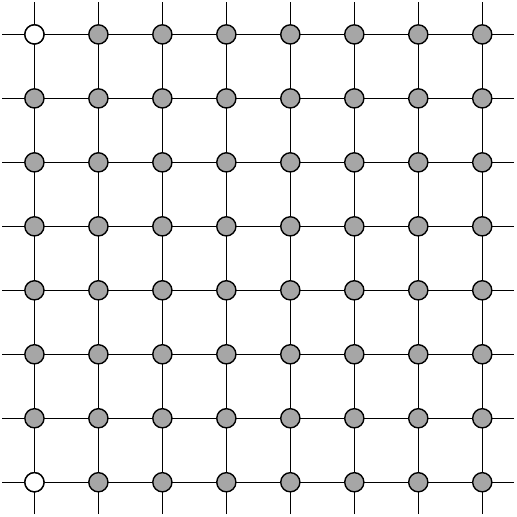}}
	\hspace{0.5cm}
	\subfloat[$\o^{(1)}_7$]{\includegraphics[scale=1]{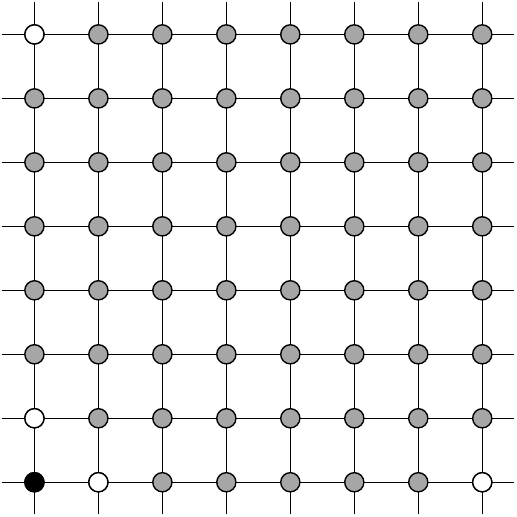}}
	\hspace{0.5cm}
	\subfloat[$\o^{(1)}_{11}$]{\includegraphics[scale=1]{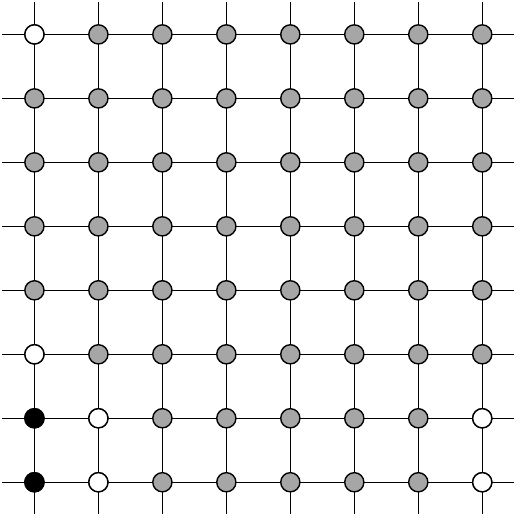}}
	\\
	\subfloat[$\s'=\o^{(1)}_{31}=\o^{(2)}_{1}$]{\includegraphics[scale=1]{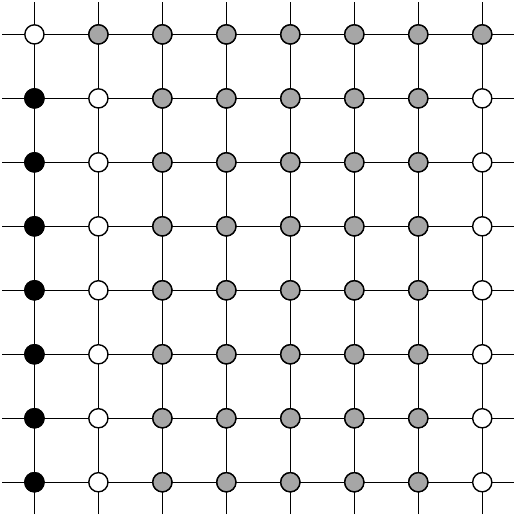}}
	\hspace{0.5cm}
	\subfloat[$\o^{(2)}_{21}$]{\includegraphics[scale=1]{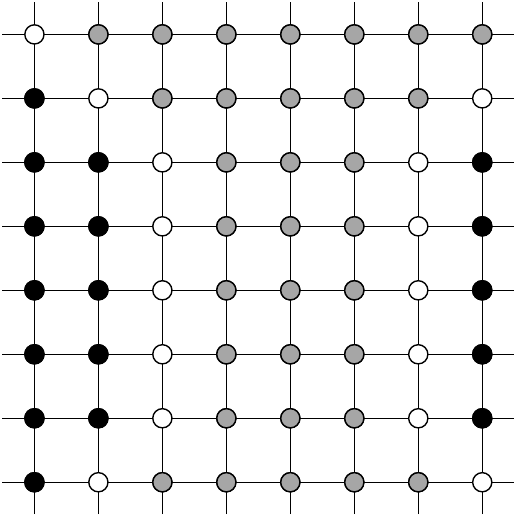}}
	\hspace{0.5cm}
	\subfloat[$\o^{(2)}_{33}$]{\includegraphics[scale=1]{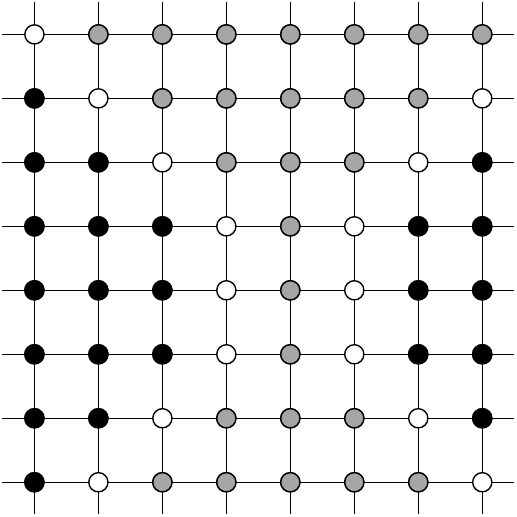}}
	\\
	\subfloat[$\o^{(2)}_{34}$]{\includegraphics[scale=1]{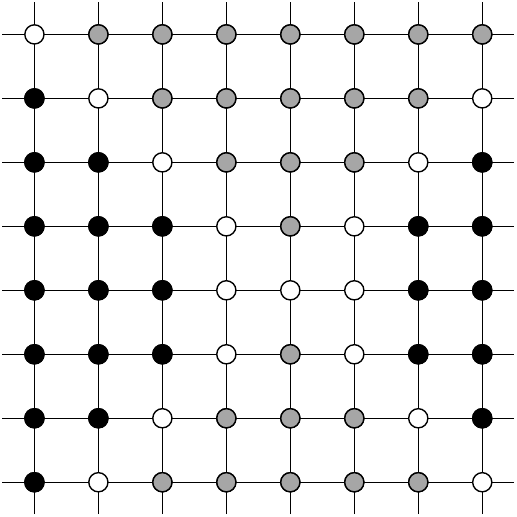}}
	\hspace{0.5cm}
	\subfloat[$\s''=\o^{(2)}_{36}$]{\includegraphics[scale=1]{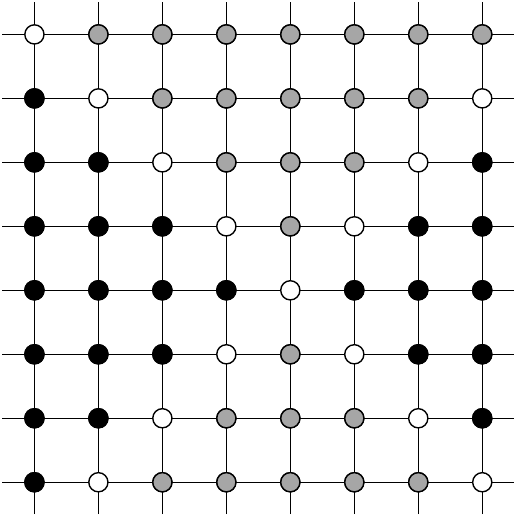}}
	\hspace{0.5cm}
	\subfloat[$\s'''$]{\includegraphics[scale=1]{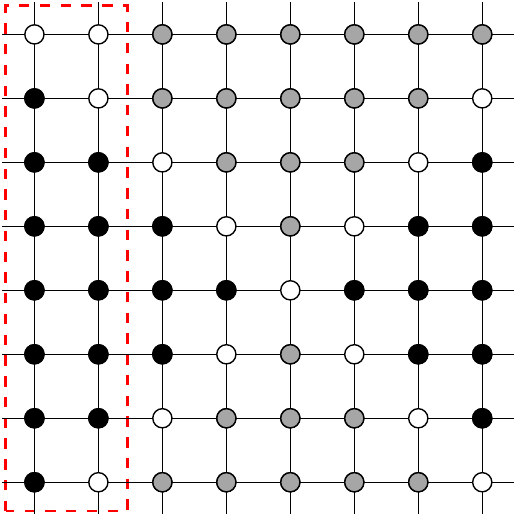}}
	\caption{Some snapshots of the paths $\o^{(1)}$ and $\o^{(1)}$ for the $8 \times 8$ square lattice $\L$ with periodic boundary conditions}
	\label{fig:refpathWRtoric_omega1}
\end{figure}
\FloatBarrier

It is easy to check that
$
	\D(\s'') = 2K-2.
$
Furthermore, it is possible to build such a black diamond starting from $\s'$ by alternatively removing a (gray) particle and adding immediately after a (black) particle, so that the path $\o^{(2)}$ satisfies
\begin{equation}
\label{eq:phiomega2_KequalL}
	\Phi_{\o^{(2)}} - H(\s') = \max_{\h \in \o^{(2)}} H(\h) - H(\s') = 1.
\end{equation}
Consider now configuration $\s'''$ obtained from $\s''$ by removing the gray particle in the site $(1,K-1)$ and thus $H(\s''')=H(\s'')+1$. Having only black particles or empty sites in stripe $C_1$, the configuration $\s'''$ satisfies the initial condition~\eqref{eq:racond_toricgrid} and thus the reduction algorithm can be used to build a path $\o^{(3)}$ from $\s'''$ to $\bb$ such that
\begin{equation}
\label{eq:phiomega3_KequalL}
	\Phi_{\o^{(3)}} - H(\s''') = \max_{\h \in \o^{(3)}} H(\h) - H(\s''') = 1.
\end{equation}
In view of~\eqref{eq:phiomega1_KequalL}-\eqref{eq:phiomega3_KequalL}, the path $\o^*$ obtained by concatenating $\o^{(1)}$, $\o^{(2)}$ and $\o^{(3)}$ satisfies
\[
	\Phi_{\o^*} - H(\aa) = \max_{\h \in \o^*} H(\h) -H(\aa) = 2K.
\]

\captionsetup[subfigure]{labelformat=empty}
\begin{figure}[!h]
	\centering
	\subfloat[$\aa$]{\includegraphics[scale=0.925]{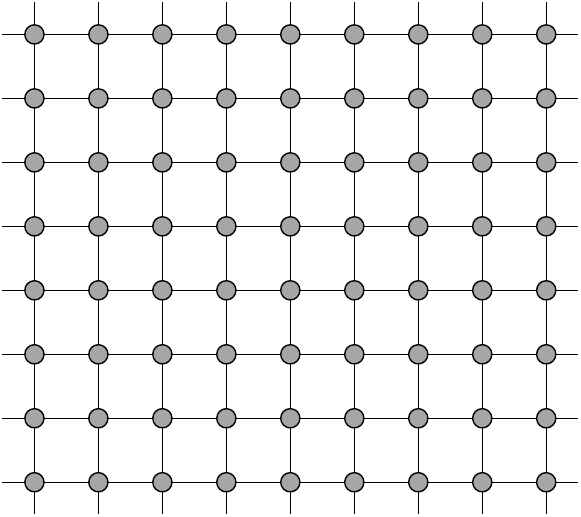}}
	\hspace{0.2cm}
	\subfloat[$\s^*$]{\includegraphics[scale=0.925]{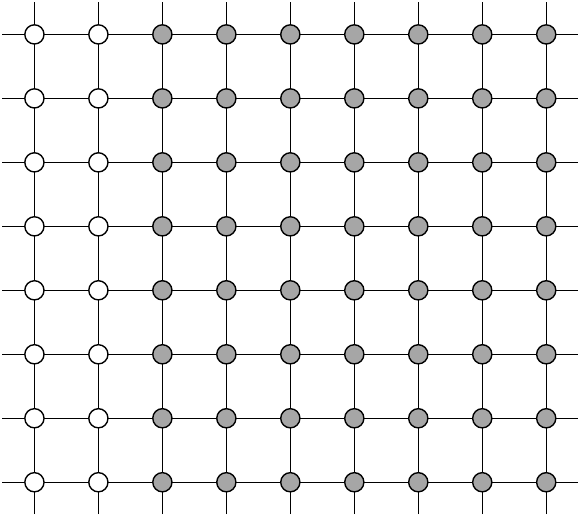}}
	\hspace{0.2cm}
	\subfloat[$\o^{(2)}_2$]{\includegraphics[scale=0.925]{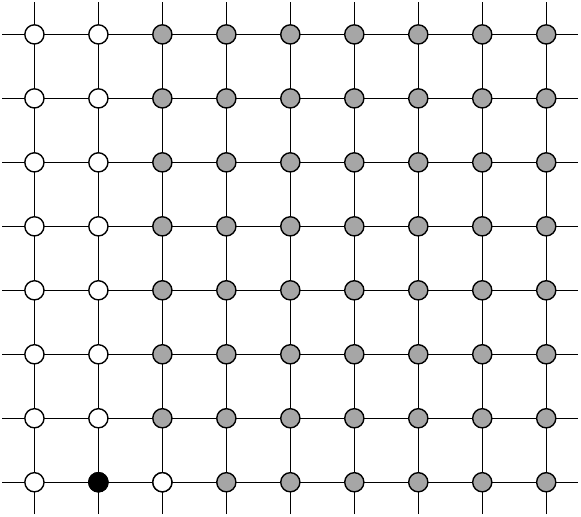}}
	\hspace{0.2cm}
	\subfloat[$\o^{(2)}_{16}$]{\includegraphics[scale=0.925]{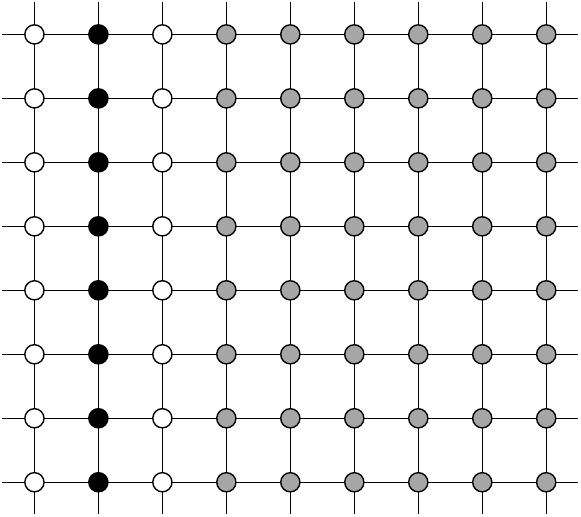}}
	\hspace{0.2cm}
	\subfloat[$\o^{(2)}_{18}$]{\includegraphics[scale=0.925]{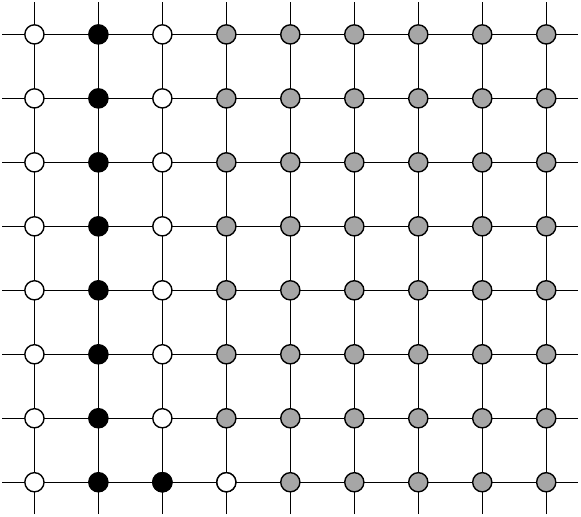}}
	\hspace{0.2cm}
	\subfloat[$\bb$]{\includegraphics[scale=0.925]{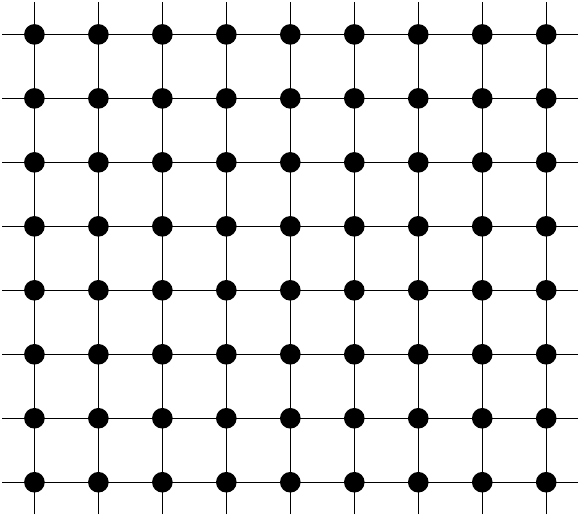}}
	\caption{Some snapshots of the reference path $\o^*:\aa\to\bb$ for the $8 \times 9$ square lattice  $\L$ with periodic boundary conditions}
	\label{fig:refpathWRtoric_2}
\end{figure}
\FloatBarrier
\captionsetup[subfigure]{labelformat=parens}

In case (b), where $K<L$, we will show that there exists a path $\o^*: \aa \to \bb$ such that $\Phi_{\o^*} - H(\aa) = 2K+1$. We construct such a path $\o^*$ as the concatenation of two shorter paths, $\o^{(1)}$ and $\o^{(2)}$, where $\o^{(1)}: \aa \to \s^*$ and $\o^{(2)}: \s^* \to \bb$. The intermediate configuration $\s^* \in \cX$ is the one obtained from $\aa$ removing all the particles residing in the first vertical stripe $C_1$ (see also Figure~\ref{fig:refpathWRtoric_2}), \ie
\[
	\s^*(v):=
	\begin{cases}
		\aa(v) 	& \text{ if } v \in \L \setminus C_1,\\
		0 			& \text{ if } v \in C_1.
	\end{cases}
\]
The path $\o^{(1)}=(\o^{(1)}_{1},\dots,\o^{(1)}_{2K+1})$, with $\o^{(1)}_{1}=\aa$ and $\o^{(1)}_{2K+1}=\s^*$ can be constructed as follows. For $i=1,\dots,2K$, at step $i$ we remove from configuration $\o^{(1)}_{i}$ the particle in site $(\lfloor \frac{i}{K} \rfloor, i- K \cdot \lfloor \frac{i}{K} \rfloor)$, increasing the energy by $1$ and obtaining in this way configuration $\o^{(1)}_{i+1}$. Therefore the configuration $\s^*$ is such that $H(\s^*)-H(\aa) = 2K$ and
\begin{equation}
\label{eq:phiomega1WR}
	\Phi_{\o^{(1)}} = H(\s^*) = H(\aa)+ 2K.
\end{equation}

The second path $\o^{(2)}: \s^* \to \bb$ is then constructed by means of the reduction algorithm, which can be used here since the configuration $\s^*$ satisfies condition~\eqref{eq:racond_toricgrid}, having no particles on the vertical stripe $C_1$. The algorithm guarantees that 
\begin{equation}
\label{eq:phiomega2WR}
	\Phi_{\o^{(2)}} = H(\s^*) +1.
\end{equation}
The concatenation of $\o^{(1)}$ and $\o^{(2)}$ yields a path $\o^*: \aa \to \bb$, that, by virtue of~\eqref{eq:phiomega1WR} and~\eqref{eq:phiomega2WR}, satisfies $\Phi_{\o^*} = H(\aa)+ 2K+1$. 
\end{proof}

The reduction algorithm is also the key ingredient in the proof Theorem~\ref{thm:structuralpropertiesWR}(ii), which we now present.

\begin{proof}[Proof of Theorem~\ref{thm:structuralpropertiesWR}(ii)]
Assume without loss of generality that $K\leq L$. We want to show that for every WR configuration $\s \not\in \ss$ the following inequality holds
\[
	\Phi(\s,\ss) - H(\s) < 2 K.
\]
We distinguish two scenarios, depending on whether (a) $\s$ has a vertical $m$-bridge for some $m=1,\dots,\mm$ or (b) $\s$ has no vertical bridges. In scenario (a) we will show that $\Phi(\s,\bs^{(m)}) - H(\s) \leq 1$, while in scenario (b) we will prove that $\Phi(\s,\bbb) - H(\s) < 2 K$, where $\bbb \in \ss$ is the stable configuration with only particles of type $b$ (to which we henceforth refer as \textit{black} particles).

Consider scenario (a) first, and, without loss of generality, assume that $\s$ has a vertical $m$-bridge on the first column $c_0$. Due to the hard-core constraints between unlike particles, there cannot be particles on column $c_1$, unless they are of type $m$. Hence, $\s$ satisfies condition~\eqref{eq:racond_toricgrid_bis} and is then a suitable initial configuration for the reduction algorithm with target state $\bs^{(m)}$. Using this procedure, we obtain a path $\o$ such that $\Phi(\s,\bs^{(m)}) - H(\s) \leq 1$, since at every step where a particle of type different $m$ is (possibly) removed, at the next step a particle of type $m$ is added in the neighboring site.

As far as scenario (b) is concerned, denote by $g$ the number of non-black particles that configuration $\s$ has in the first vertical stripe. By construction $\s$ has no vertical bridges and thus
\begin{equation}
\label{eq:tb_toric}
	g \leq 2K-2.
\end{equation}
Let $\s^* \in \cX$ be the configuration obtained from $\s$ by removing all the $g$ non-black particles in the first vertical stripe, namely
\[
	\s^*(v):=
	\begin{cases}
		\s(v) 	& \text{ if } v \in \L \setminus C_1 \text{ or } v \in C_1 \text{ and } \s(v)=b,\\
		0 			& \text{ if } v \in C_1 \text{ and } \s(v)\neq b.
	\end{cases}
\]
Clearly $H(\s^*)-H(\s) = g$. We construct a path  $\o^{(1)}=(\o^{(1)}_1,\dots,\o^{(1)}_{g+1})$ from $\o^{(1)}_1=\s$ to $\o^{(1)}_{g+1}=\s^*$ as follows. For $i=1,\dots,g$, at step $i$ we remove the first non-black particle in lexicographic order residing in stripe $C_1$ from configuration $\o^{(1)}_i$, obtaining in this way $\o^{(1)}_{i+1}$, which is such that $H(\o^{(1)}_{i+1})= H(\o^{(1)}_{i})+1$. Thus,
\begin{equation}
\label{eq:phiomega1}
	\Phi_{\o^{(1)}} = H(\s^*) = H(\s)+g.
\end{equation}
Note that the configuration $\s^*$ satisfies condition~\eqref{eq:racond_toricgrid}, since it all sites in the first vertical stripe $C_1$ are either empty or occupied by black particles. Thus, $\s^*$ is a suitable initial configuration for the reduction algorithm, which returns a path $\o^{(2)}: \s^* \to \bbb$ such that
\begin{equation}
\label{eq:phiomega2}
	\Phi_{\o^{(2)}} \leq H(\s^*)+1.
\end{equation}
The concatenation of paths $\o^{(1)}$ and $\o^{(2)}$ yields a new path $\o: \s \to \bbb$ that, in view of~\eqref{eq:phiomega1} and~\eqref{eq:phiomega2}, satisfies inequality $\Phi_{\o} \leq H(\s)+g+1$. Therefore, using~\eqref{eq:tb_toric}, we get
\[
	\Phi(\s,\bbb)-H(\s) \leq g+1 < 2K. \qedhere
\]
\end{proof}

In both Propositions~\ref{prop:lowertg_WR} and~\ref{prop:refpathwrt} we excluded the special cases $(K,L) =(2,2),(2,3),(3,2)$, in which it can be checked ``by hand'' that the communication height between stable configurations is still the highest energy barrier of the entire energy landscape and that takes the values $3$, $4$, and $4$, respectively.

\section{Proofs for square lattices with open boundary conditions}
\label{sec6}
In this section we prove the structural properties outlined in Theorem~\ref{thm:structuralpropertiesWR} for the energy landscape corresponding to the Widom-Rowlison model on a square lattice $\L$ with \textit{open} boundary conditions. The proof approach is the same as that used in Section~\ref{sec5} for square lattices with periodic boundary conditions, but a few details and counting arguments need to be adjusted appropriately, in view of the different structure that rows and columns have in this case, as one notice by comparing Lemmas~\ref{lem:bqbi} and~\ref{lem:bqbso_opengrid}. 

This section is organized as follows. We first prove a lower bound for the communication height between any pair of stable configurations, see Proposition~\ref{prop:lowerog_WR} below. We then introduce a modified version of \textit{reduction algorithm} introduce in Section~\ref{sec5} that leverages the open boundary conditions and use it to exhibit a reference path between any pair of stable configurations (Proposition~\ref{prop:refpathwro}) that attains the lower bound in Proposition~\ref{prop:lowerog_WR}, completing in this way the proof of Theorem~\ref{thm:structuralpropertiesWR}(i). Lastly, we use again the reduction algorithm to construct paths with a prescribed energy height from every configuration $\s \not \in \ss$ to the subset $\ss$, proving Theorem~\ref{thm:structuralpropertiesWR}(ii).

\begin{prop}[Lower bound for $\Phi(\aa,\bb)$]\label{prop:lowerog_WR}
Consider the Widom-Rowlison model on the $K \times L$ square lattice $\L$ with open boundary conditions. The communication height between any pair of stable configurations $\aa, \bb \in \ss$ in the corresponding energy landscape satisfies
\[
	\Phi(\aa,\bb) - H(\aa) \geq \min\{K,L\}+1.
\]
\end{prop}
\begin{proof}
Similarly to what has been done in the proof of Proposition~\ref{prop:lowertg_WR}, also for this proof we can work under the assumption that $\mm=2$. Indeed, using the same automorphism $\hat{\xi}$ of the space space $\cX$ introduced in~\eqref{eq:projection}, every WR configuration $\s$ with $\mm$ particle types is mapped to a new WR configuration $\xi(\s)$ that, while having only two types of particles, has the same energy as the original configuration $\s$. In particular, every path $\o: \bs \to \bsp$ is mapped to a path $\hat{\xi}(\o): \bs \to \bsp$ with the same energy profile and height, since  
\[
	H(\hat{\xi}(\o)_i) = H(\hat{\xi}(\o_i)) = H(\o_i) , \quad \forall \, i =1,\dots,n.
\]
We henceforth assume that $\mm=2$ and, modulo a relabeling of the particle types, we can assume that $\bs=\bs^{(1)}$ and $\bsp=\bs^{(2)}$ and associate the gray color to particles of type $1$ and the black color to those of type $2$.\\

Furthermore, without loss of generality, we may assume that $K \leq L$. We need to show that in every path $\o: \aa \to \bb$ there is at least one configuration with energy difference greater than or equal to $K+1$. Take a path $\o = (\o_1,\dots, \o_n)$ and, without loss of generality, we may assume that there are no void moves in $\o$, \ie at every step either a particle is added or a particle is removed, so that $H(\o_{i+1}) = H(\o_{i}) \pm 1$ for every $1 \leq i \leq n-1$. Since $\aa$ has no black bridges, while $\bb$ does, at some point along the path $\o$ there must be a configuration $\o_{n^*}$ that is \textit{the first} to display a black vertical bridge or a black horizontal bridge or both simultaneously (i.e.~a black cross). We claim that 
\[
	\max \{\D(\o_{n^*-1}), \D(\o_{n^*-2})\} \geq \min\{K,L\}+1.
\]
We distinguish the following three cases:
\begin{itemize}
	\item[(a)] $\o_{n^*}$ displays a black vertical bridge only;
	\item[(b)] $\o_{n^*}$ displays a black horizontal bridge only;
	\item[(c)] $\o_{n^*}$ displays a black cross.
\end{itemize}
These three cases cover all the possibilities, since the addition of a single particle cannot create two parallel bridges simultaneously.

For case (a), we claim that the energy difference in every row is greater than or equal to one. Suppose by contradiction that there exists a row $r^*$ such that $\D^{r^*}(\o_{n^*})=0$. Then, by Lemma~\ref{lem:bqbso_opengrid}, $\o_{n^*}$ would have a bridge in row $r^*$. Such a bridge cannot be black, since otherwise $\o_{n^*}$ would have a black cross, and neither gray, which by Lemma~\ref{lem:impossible} could not coexist with the black vertical bridge that $\o_{n^*}$ has, contradiction. Hence $\D^{r}(\o_{n^*})\geq 1$ for every row $r$ and thus
\[
	\D(\o_{n^*}) = \sum_{i=0}^{K-1} \D^{r_i}(\o_{n^*}) \geq K.
\]
The previous configuration $\o_{n^*-1}$ along the path $\o$ differs from $\o_{n^*}$ in exactly one site, say $v^*$, which is such that $\o_{n^*-1}(v^*)=0$ and $\o_{n^*}(v^*)=\bb(v^*)$. Hence $\D(\o_{n^*-1})=\D(\o_{n^*})+1$ and therefore
\[
	\D(\o_{n^*-1}) \geq K+1.
\]

For case (b) we can argue as in case (a), but interchanging the role of rows and columns, and obtain that 
\[
	\D(\o_{n^*-1}) \geq L +1 \geq K+1.
\]

For case (c), let $r^*$ and $c^*$ be respectively the row and the column where the black cross lies in configuration $\o_{n^*}$. The previous configuration $\o_{n^*-1}$ along the path $\o$ differs from $\o_{n^*}$ in exactly one site, denoted by $v^*$, which has to be the site where row $r^*$ and column $c^*$ intersect (otherwise configuration $\o_{n^*-1}$ would have a black bridge, violating the definition of $n^*$). Furthermore, $\D^{c^*}(\o_{n^*-1}) \geq 1$, since $\o_{n^*-1}(v^*)=0$ by construction. We claim that the energy difference in every column $c\neq c^*$ is also greater than or equal to one for configuration $\o_{n^*-1}$, namely
\begin{equation}
\label{eq:wgtr2}
	\D^{c}(\o_{n^*-1}) \geq 1, \quad \forall \, c\neq c^*.
\end{equation}
These inequalities follow from Lemma~\ref{lem:bqbso_opengrid} after noticing that in each of these $L-1$ columns configuration $\o_{n^*-1}$ cannot display neither a black vertical bridge (by definition of $n^*$), nor a gray vertical bridge, since every column $c \neq c^*$ has at least one black particle (at the intersection with row $r^*$). Summing the energy difference of all columns we get
\begin{equation}
\label{eq:wgtr2lm1}
	\D(\o_{n^*-1}) = \sum_{j=0}^{L-1} \D^{c_j}(\o_{n^*-1}) \geq L.
\end{equation} 
If $\D(\o_{n^*-1}) \geq L+1$, then the proof is complete. If instead $\D(\o_{n^*-1}) = L$, consider the configuration $\o_{n^*-2}$ preceding $\o_{n^*-1}$ in the path $\o$. By construction, the configuration $\o_{n^*-2}$ differs from $\o_{n^*-1}$ by a single-site update and thus
\begin{equation}
\label{eq:pm1}
	\D(\o_{n^*-2})=\D(\o_{n^*-1})\pm1.
\end{equation}
Consider the case where $\D(\o_{n^*-2}) = \D(\o_{n^*-1})-1 = L-1$. In this case, thanks to the pigeonhole principle, the configuration $\o_{n^*-2}$ must have either \textit{at least} one column $c$ with $\D^c(\o_{n^*-2})=0$. By Lemma~\ref{lem:bqbso_opengrid}, $\o_{n^*-2}$ must have a bridge in this column, see Figure~\ref{fig:nogb}. However it cannot be a black bridge, due to the definition of $n^*$, and neither a gray bridge, since it would not be possible to obtain a black bridge in row $r_*$ (which $\o_{n^*}$ has) with only two admissible single-site updates (at least five steps are needed, as shown in Figure~\ref{fig:nogb}).
\begin{figure}[!h]
	\centering
	\includegraphics[scale=0.95]{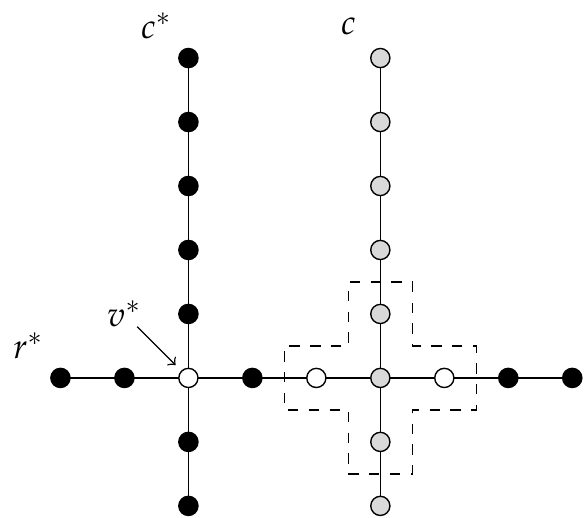}
	\caption{The dashed line encloses the $5$ sites that should be updated if configuration $\o_{n^*-2}$ had a gray bridge in column $c$}
	\label{fig:nogb}
\end{figure}
\FloatBarrier
\noindent Therefore, it is not possible that $\D(\o_{n^*-2}) = \D(\o_{n^*-1})-1$, and combining~\eqref{eq:wgtr2lm1} and~\eqref{eq:pm1} yields
\[
	\D(\o_{n^*-2})=\D(\o_{n^*-1}) + 1 = L+1. \qedhere
\]
\end{proof}

\subsection*{Reduction algorithm for WR configurations (open boundary conditions)}
We now briefly illustrate how the reduction algorithm presented in Section~\ref{sec5} should be adjusted when $\L$ is a square lattice with open boundary conditions. Assume that we want to build a path from a WR configuration $\s \in \cX$ to a target stable configuration, say $\bbb \in \ss$, consisting of particles of type $b$ to which we simply refer as \textit{black particles}. The underlying idea is the same as that of the reduction algorithm presented in Section~\ref{sec5} for square lattices with periodic boundary conditions: The procedure yields a path in which black particles are progressively added column by column to the original configuration $\s$ removing all non-black particles when necessary.

In order to be able to start this procedure, a suitable initial configuration $\s$ for this iterative procedure needs to satisfy a condition weaker than~\eqref{eq:racond_toricgrid}, thanks to the open boundary conditions of $\L$: Indeed, we require only that the configuration $\s$ has all the sites in the first column $c_0$ either empty or occupied by black particles, \ie
\begin{equation}
\label{eq:racond_opengrid}
	\s(v) \in \{0,b\} \quad \forall \, v \in c_0.
\end{equation}
The algorithm we are about to describe uses in a crucial way the open boundary conditions of $\L$ and, in particular, the fact that all the sites in the first column $c_0$ have no left neighboring sites. For this reason, the procedure yields a path with a lower energy height than in the case of a square lattice with periodic boundary conditions.

The path $\o: \s \to \bbb$ is built as the concatenation of $L$ paths $\o^{(1)}, \dots, \o^{(L)}$. Path $\o^{(j)}$ goes from $\s_{j}$ to $\s_{j+1}$, where we define $\s_1:=\s$ and, for $j=2,\dots,L+1$,
\[
	\s_{j}(v) :=
	\begin{cases}
		b 	& \text{ if } v \in \bigcup_{i=0}^{j-2} c_j,\\
		0 & \text{ if } v \in c_{j-1} \text{ and } \s(v) \neq b,\\
		\s(v) & \text{ if } v \in c_{j-1} \text{ and } \s(v)= b \text{ or } v \in \bigcup_{i=j}^{L-1} c_j.
	\end{cases}
\]
It can be checked that indeed $\s_{L+1}=\bbb$. For $j=1,\dots,L$ the path $\o^{(j)}=(\o^{(j)}_1, \dots, \o^{(j)}_{2K+1})$ consist of $2K+1$ moves (some of them possibly void), with $\o^{(j)}_1=\s_j$ and $\o^{(j)}_{2K+1}=\s_{j+1}$. We repeat iteratively the following procedure for every $i=1,\dots,2K$:
\begin{itemize}
\item If $i \equiv 1 \pmod 2$, consider site $v=(j, (i-1)/2)$.
	\begin{itemize}
		\item[-] If $\o^{(j)}_i(v) \in \{0,b\}$, then set $\o^{(j)}_{i+1}(v)=\o^{(j)}_{i}(v)$.
		\item[-] If $\o^{(j)}_i(v) \not  \in \{0,b\}$, then remove the non-black particle in site $v$ from configuration $\o^{(j)}_i$, obtaining in this way a new configuration $\o^{(j)}_{i+1}$ with $H(\o^{(j)}_{i+1}) = H(\o^{(j)}_{i})+1$.
	\end{itemize}
\item If $i \equiv 0 \pmod 2$, consider site $v=(j-1,i/2-1)$.
	\begin{itemize}
		\item[-] If $\o^{(j)}_i(v)=b$, then set $\o^{(j)}_{i+1}(v)=\o^{(j)}_i(v)=b$.
		\item[-] If $\o^{(j)}_i(v)=0$, then add a black particle in site $v$, obtaining in this way a new configuration $\o^{(j)}_{i+1}$ such that $\o^{(j)}_{i+1}(v)=b$ and thus $H(\o^{(j)}_{i+1}) = H(\o^{(j)}_{i})-1$. The new configuration does not violated the WR constraints, since, by construction, all the neighboring sites of $v$ are either empty or occupied by black particles. In particular, the site at the right of $v$ has possibly been emptied in the previous step.
	\end{itemize}
\end{itemize}
Note that for the last path $\o^{(L)}$ all steps corresponding to odd values of $i$ are void, since there is no column $c_{L}$. 

In words, the reduction algorithm alternately removes a non-black particle (if any) and adds a black particle, progressively column by column. For every $j=1, \dots, L$, the configurations $\s_j$ and $\s_{j+1}$ satisfy
\[
	H(\s_{j+1}) \leq H(\s_j),
\]
Indeed, by looking at the way the path $\o^{(j)}$ connecting them is defined, the number of black particles added in column $c_{j-1}$ is greater than or equal to the number of gray particles removed in column $c_{j}$. Moreover,
\[
	\Phi_{\o^{(j)}} \leq H(\s_j) + 1,
\]
since along the path $\o^{(j)}$ every particle removal (if any) is always followed by a particle addition. The latter two inequalities imply that the path $\o: \s \to \bbb$ created by concatenating $\o^{(1)}, \dots, \o^{(L)}$ satisfies
\[
	\Phi_{\o} \leq H(\s)+1,
\]
which shows that $\Phi(\s,\bbb)-H(\s) \leq 1$.\\

We remark that, with a few minor tweaks, we can define a similar reduction algorithm that builds a path $\o$ from a configuration $\s$ to the stable configuration $\bs^{(m)}$ such that $\Phi_{\o} \leq H(\s)+1$, provided that the initial configuration has either empty sites or particles of type $m$ on the first column, \ie
\begin{equation}
\label{eq:racond_opengrid_bis}
	\s(v) \in \{0,m\} \quad \forall \, v \in c_0.
\end{equation}

The next proposition uses the reduction algorithm to exhibit a reference path between any pair of stable configurations $\aa,\bb \in \ss$ with a prescribed energy height. 

\begin{prop}[Reference path between stable WR configurations]\label{prop:refpathwro}
Consider the energy landscape corresponding to the Metropolis dynamics of the Widom-Rowlinson model on a square lattice $\L$ with open boundary conditions. Then, for any pair of stable configurations $\aa,\bb \in \ss$ there exists a path $\o^*: \aa \to \bb$ in $\cX$ such that
\[
	\Phi_{\o^*} - H(\aa) = \min\{K,L\}+1.
\]
\end{prop}
\begin{proof} Without loss of generality, we may assume $K \leq L$ and show that there exists a path $\o^*: \aa \to \bb$ such that $\Phi_{\o^*} - H(\aa) = K +1$. We describe just briefly how the reference path $\o^*$ is constructed, since it is very similar to the one given in case (b) of the proof of Proposition~\ref{prop:refpathwrt}. Figure~\ref{fig:refpathWRopen} shows some snapshots of the reference path for a $8 \times 8$ square lattice using again the convention that the particles types in $\aa$ and $\bb$ are colored in gray and black, respectively. The path $\o^*$ is the concatenation of two shorter paths, $\o^{(1)}$ and $\o^{(2)}$, where $\o^{(1)}: \aa \to \s^*$ and $\o^{(2)}: \s^* \to \bb$, where $\s^*$ is the WR configuration that differs from $\aa$ only by having all sites of the leftmost column $c_0$ empty, \ie
\[
	\s^*(v):=
	\begin{cases}
		\aa(v) & \text{ if } v \in \L \setminus c_0,\\
		0 & \text{ if } v \in c_0.
	\end{cases}
\]
The path $\o^{(1)}$ consists of $K$ steps, at each of which we remove the first gray particle in $c_0$ in lexicographic order from the previous configuration. The last configuration is precisely $\s^*$, which has energy $H(\s^*)=H(\aa)+ K$, and, trivially, $\Phi_{\o^{(1)}} = H(\aa)+ K$. The second path $\o^{(2)}: \s^* \to \bb$ is then constructed by means of the reduction algorithm, which can be used since configuration $\s^*$ is a suitable initial configuration for it, satisfying condition~\eqref{eq:racond_opengrid}. The algorithm guarantees that $\Phi_{\o^{(2)}} = H(\s^*) +1$ and thus the concatenation of the two paths $\o^{(1)}$ and $\o^{(2)}$ yields a path $\o^*$ with $\Phi_{\o^*} = H(\aa) + K +1$.
\end{proof}
\captionsetup[subfigure]{labelformat=empty}
\begin{figure}[!h]
	\centering
	\subfloat[$\aa$]{\includegraphics[scale=1]{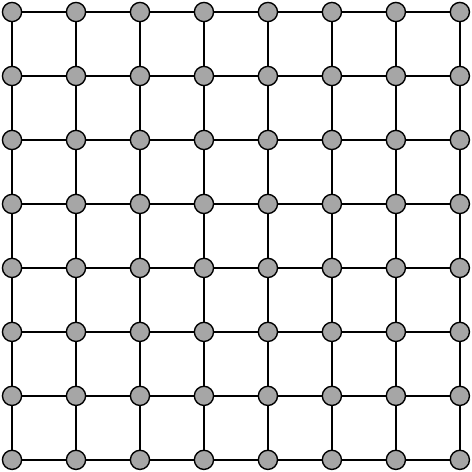}}
	\hspace{0.5cm}
	\subfloat[$\s^*$]{\includegraphics[scale=1]{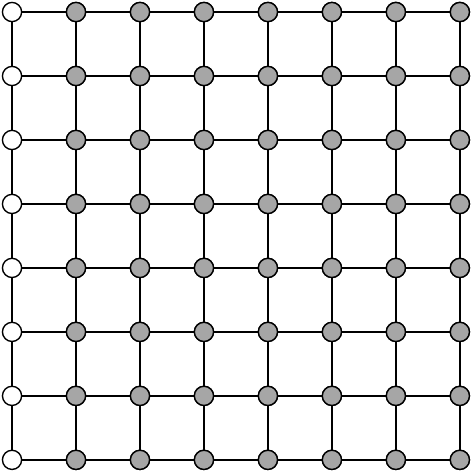}}
	\hspace{0.5cm}
	\subfloat[$\o^{(2)}_2$]{\includegraphics[scale=1]{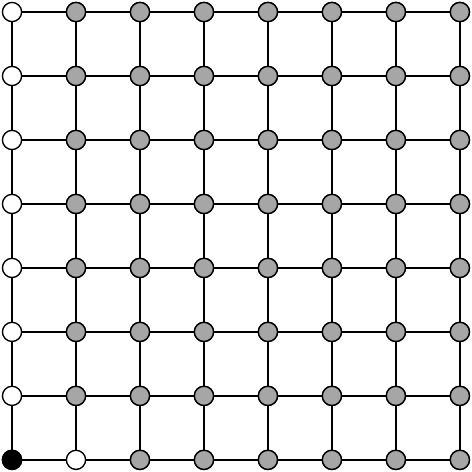}}
	\\
	\subfloat[$\o^{(2)}_{16}$]{\includegraphics[scale=1]{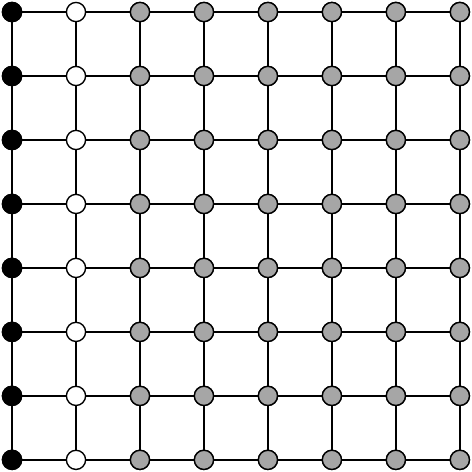}}
	\hspace{0.5cm}
	\subfloat[$\o^{(2)}_{18}$]{\includegraphics[scale=1]{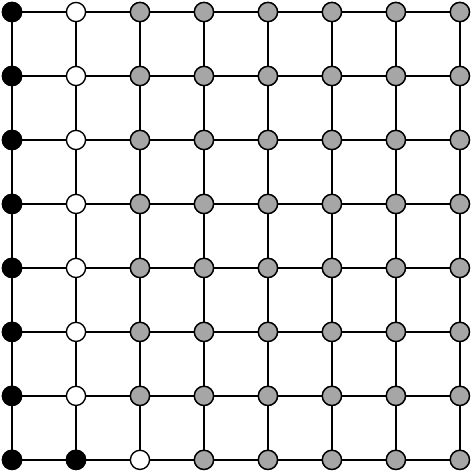}}
	\hspace{0.5cm}
	\subfloat[$\bb$]{\includegraphics[scale=1]{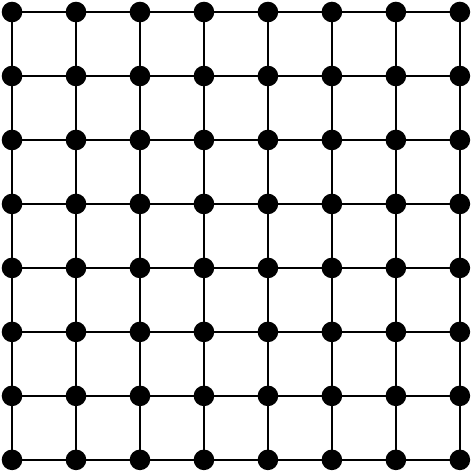}}
	\caption{Some snapshots of the reference path from $\aa$ to $\bb$ for the $8\times 8$ square lattice with open boundary conditions}
	\label{fig:refpathWRopen}
\end{figure}
\FloatBarrier

\begin{proof}[Proof of Theorem~\ref{thm:structuralpropertiesWR}(i)-(ii)]
For the Widom-Rowlison model on a square lattice $\L$ with open boundary conditions, statement (i) is an immediate consequence of the lower bound for the communication height between stable configurations given in Proposition~\ref{prop:lowerog_WR} and the matching upper bound given by the reference path exhibited in Proposition~\ref{prop:refpathwro}.

We now prove statement (ii) using once more the reduction algorithm described earlier in this section. The goal is to show that for every configuration $\s  \not\in \ss$ the inequality $\Phi(\s,\ss) - H(\s) \leq \min\{K,L\}$ holds. Since we may assume, without loss of generality, that $K \leq L$, it is enough to prove that
\[
	\Phi(\s,\ss) - H(\s) \leq K, \quad \forall \, \s  \not\in \ss.
\]
We distinguish two cases, depending on whether (a) $\s$ has a vertical $m$-bridge on the first column $c_0$ for some $m =1,\dots,\mm$ or (b) $\s$ has no vertical bridge on $c_0$.

In case (a), $\s$ has only particles of type $m$ on column $c_0$ and, as such, is a suitable starting configuration for the reduction algorithm with target state $\bs^{(m)}$, since condition~\eqref{eq:racond_opengrid_bis} is satisfied. Hence, we can build a path $\o: \s \to \bs^{(m)}$ such that $\Phi_{\o} \leq H(\s)+1$, showing that in this case $\Phi(\s,\bs^{(m)}) - H(\s) \leq 1$.

Consider now case (b). Since there is no vertical bridge on the first column $c_0$, Lemma~\ref{lem:bqbso_opengrid} implies that $\s$ has at most $K-1$ particles on $c_0$; denote by $g$ their number. In this case we create a path $\o$ from $\s$ to any stable configuration $\bs \in \ss$ as the concatenation of two shorter paths, $\o^{(1)}$ and $\o^{(2)}$, where $\o^{(1)}: \s \to \s^*$, $\o^{(2)}: \s^* \to \bs$ and the intermediate WR configuration $\s^*$ is the one obtained from $\s$ by removing the particles residing in the leftmost column $c_0$, \ie
\[
	\s^*(v):=
	\begin{cases}
		\s(v) 	& \text{ if } v \in \L \setminus c_0,\\
		0 			& \text{ if } v \in c_0.
	\end{cases}
\]
The path $\o^{(1)}$ can be easily defined by progressively removing all $g$ particles in column $c_0$, increasing the energy by $1$ at each step. Therefore, the configuration $\s^*$ is such that $H(\s^*)-H(\s) = g$ and the following inequality holds:
\[
	\Phi_{\o^{(1)}} \leq H(\s^*) = H(\s)+g.
\]
The path $\o^{(2)}: \s^* \to \bs$ is then constructed by means of the reduction algorithm described earlier, using $\s^*$ as initial configuration (condition~\eqref{eq:racond_opengrid_bis} is satisfied for any $m$) and $\bs$ as target configuration. This procedure guarantees that
\[
	\Phi_{\o^{(2)}} \leq H(\s^*) + 1.
\]
The concatenation of the two paths $\o^{(1)}$ and $\o^{(2)}$ then gives a path $\o: \s \to \bs$ that satisfies the inequality $\Phi_{\o} \leq H(\s)+ g + 1$ and, since $g < K$ by construction, we obtain
\[
	\Phi(\s,\ss)-H(\s) \leq \Phi(\s,\bs)-H(\s) = g+1 \leq K. \qedhere
\]
\end{proof}

\textbf{Acknowledgments} The author is supported by NWO grant 639.033.413. The author is grateful to F.R.~Nardi, S.C.~Borst, and  J.S.H.~van Leeuwaarden for the precious feedback and helpful discussions related to this work. 

\bibliographystyle{plain}
\bibliography{library}
\end{document}